\documentclass[9pt]{amsart}
\usepackage{amsmath,amssymb}
\usepackage{graphicx,xcolor}
\usepackage[colorlinks, citecolor=blue]{hyperref}
\usepackage{color}
\usepackage[all]{xy}
\usepackage{amsmath,amssymb,cite}
\usepackage{mathrsfs}
\usepackage{amsthm}
\usepackage{booktabs}
\usepackage{tabularx}
\usepackage{colortbl,booktabs}
\usepackage{tabu}
\usepackage{multicol}
\usepackage{multirow}
\usepackage{threeparttable}
\usepackage{CJK}
\newtheorem{thm}{Theorem}
\newtheorem{prop}{Proposition}
\newtheorem{lem}{Lemma}
\theoremstyle{definition}
\newtheorem{ex}{Example}
\newtheorem{defi}{Definition}
\newtheorem{cor}{Corollary}
\newtheorem{prob}{Problem}

\theoremstyle{remark}
\newtheorem{rem}{Remark}

\begin{document}
\title{The language of pre-topology in knowledge spaces}

\author{Fucai Lin$^{*1, 2}$, Xiyan Cao$^1$, Jinjin Li$^1$}

\thanks{*Corresponding author.
\\This work is supported by the Key Program of the Natural Science Foundation of Fujian Province (No: 2020J02043) and the National
Natural Science Foundation of China (Grant No. 11571158).
\\E-mail address: linfucai@mnnu.edu.cn}
\maketitle
 \begin{center}{\small $^1$ School of Mathematics and Statistics, Minnan Normal University, Zhangzhou 363000, China\\ }
{\small $^2$ Fujian Key Laboratory of Granular Computing and Applications,
 Minnan Normal University, Zhangzhou 363000, P. R. China\\}
\end{center}
{\baselineskip 0.3in
\noindent {\bf Abstract} We systematically study some basic properties of the theory of pre-topological spaces, such as, pre-base, subspace, axioms of separation, connectedness, etc. Pre-topology is also known as knowledge space in the theory of knowledge structures. We discuss the language of axioms of separation of pre-topology in the theory of knowledge spaces, the relation of Alexandroff spaces and quasi ordinal spaces, and the applications of the density of pre-topological spaces in primary items for knowledge spaces. In particular, we give a characterization of a skill multimap such that the delineate knowledge structure is a knowledge space, which gives an answer to a problem in \cite{falmagne2011learning} or \cite{XGLJ} whenever each item with finitely many competencies; moreover, we give an algorithm to find the set of atom primary items for any finite knowledge spaces.



\section{Introduction and preliminaries}
This paper is stimulated by a recent paper Danilov (2009), where Danilov discussed the knowledge spaces based on the topological point of view. Indeed, the notion of a knowledge space is a generalization of topological spaces \cite{2009Danilov}. Doignon
and Falmagne (1999) introduced the theory of knowledge spaces (KST) which is regarded as a mathematical framework for the assessment of knowledge and advices for further learning \cite{doignon1985spaces,falmagne2011learning}. KST makes a dynamic evaluation process; of course, the accurate dynamic evaluation is based on individuals' responses to items and the quasi order on domain $Q$\cite{doignon1985spaces}.

A field of {\it knowledge} is a non-empty set of items or questions, denoted by $Q$. A subset $H$ of $Q$ is called a {\it knowledge state} whether an individual is capable of solving it in ideal conditions. A collection $\mathscr{H}$ of knowledge state is called {\it knowledge structure} if $\{\emptyset, Q\}\subseteq\mathscr{H}$, denoted by $(Q, \mathscr{H})$.
Sometimes we simply say that $\mathscr{H}$ is the knowledge structure if the domain can be omitted without ambiguity. There are two important types of knowledge structures, namely, knowledge spaces and learning spaces. If $\mathscr{H}^{\prime}\subseteq\mathscr{H}$
implies $\bigcup\mathscr{H}^{\prime}\in \mathscr{H}$, then the knowledge structure $\mathscr{H}$ is called a {\it knowledge
space}. Moreover, a special knowledge structure $\mathscr{H}$ is called {\it learning space} if $\mathscr{H}$ satisfies learning smoothness and learning consistency, see \cite{cosyn2009note, falmagne2011learning}. Moreover, each learning space must be finite.

Let $\mathcal{F}\subseteq 2^{Q}$, and let $\mathcal{F}_{t}=\{K\in\mathcal{F}: t\in Q\}$ for each $t\in Q$. For each $t\in Q$, put $$t^{\ast}=\{r\in Q| \mathscr{H}_{r}=\mathscr{H}_{t}\};$$ the $t^{\ast}$ is called a {\it notion}. Therefore, it follows that $$t^{\ast}=\{r\in Q| \mathscr{H}_{r}=\mathscr{H}_{t}\}.$$ Moreover, if $t^{\ast}$ is single item for each $t\in Q$, then $(Q, \mathscr{H})$ is said to be {\it discriminative}. In \cite{XGLJ}, a knowledge structure $(Q, \mathscr{H})$ is called {\it bi-discriminative} if $\mathscr{H}_{t}\nsubseteq \mathscr{H}_{r}$ and $\mathscr{H}_{r}\nsubseteq \mathscr{H}_{t}$ for any distinct $t, r\in Q$. For each $L\in\mathscr{H}$, put $$L^{\ast}=\{t^{\ast}: t\in L\},$$$$\mathscr{H}^{\ast}=\{L^{\ast}: L\in\mathscr{H}\},$$and $$Q^{\ast}=\{t^{\ast}: t\in Q\}.$$ Then the knowledge structure $(Q^{\ast}, \mathscr{H}^{\ast})$ is discriminative which is produced from $(Q, \mathscr{H})$.
We say that $(Q^{\ast}, \mathscr{H}^{\ast})$ is the {\it discriminative reduction} of
$(Q, \mathscr{H})$. A knowledge space is called a {\it quasi ordinal space} if it is closed under intersection. A discriminative and quasi ordinal space is called an {\it ordinal space}. For a detailed description of KST, the reader may refer to Falmagne and Doignon\cite{doignon1985spaces,doignon1997well,doignon1999knowledge,doignon2011knowledge,falmagne1988markovian,falmagne1989latent,falmagne1990introduction,falmagne2011learning,falmagne2013knowledge}.

The main motivation of the this paper is systemic to study the theory of pre-topological spaces, and to describe the language of the theory of pre-topology in knowledge spaces. This paper is organized as follows. Section 1 presents some relevant background about KST. We systematically introduce the theory of pre-topological spaces in section 2. The applications of the theory of pre-topological spaces in knowledge spaces are studied in section 3. Section 4 summarizes the main results of this paper.

Denote the sets of real number, rational number, positive integers, the closed unit interval and all non-negative integers by $\mathbb{R}$, $\mathbb{Q}$, $\mathbb{N}$, $I$ and $\omega$, respectively. For any sets $A$ and $B$, we denote $(A\cup B)\setminus(A\cap B)$ by $A\bigtriangleup B$. Readers may refer
\cite{E1989,falmagne2011learning} for terminology and notations not
explicitly given here.

\section{The theory of pre-topology}
It is well known that Cs\'{a}sz\'{a}r (2002) in \cite{Generalized2002} introduced the notions of generalized
topological spaces. A {\it generalized topology} on a set $Z$ is a subfamily $\mathscr{T}$ of $2^{Z}$ such that $\mathscr{T}$ is closed
under arbitrary unions. Clearly, $\emptyset\in\mathscr{T}$. Short introductions to the theory of generalized topology are contained in \cite{Generalized2002,Generalized2004,Generalized2008,Generalized2009}. In this paper, we are interested in a special generalized topological spaces, that is, pre-topological spaces. Indeed, J. Li first discuss the pre-topology (that is, the subbase for the topology) with the applications in rough sets, see \cite{lijinjin2004,lijinjin2006,lijinjin2007}. Then D. Liu in \cite{liudejin2011,liudejin2013} discuss some properties of pre-topology.

In this section, we are interesting pre-topology in two aspects. On one hand, we will unify the terms of the Cs\'{a}sz\'{a}r's and Li's such that the theory of pre-topology is well applied in knowledge spaces. On the other hand, in order to give the integrity of the whole theoretical framework of the theory of pre-topology, we shall systematically list and give some propositions and theorems without proofs except for some part of new results; the reader can find the similar proofs in \cite{E1989} for some results.

\begin{defi}\cite{Generalized2002,lijinjin2004, 2009Danilov}
A {\it pre-topology} on a set $Z$ is a subfamily $\mathscr{T}$ of $2^{Z}$ such that $\bigcup\mathscr{T}=Z$ and $\bigcup\mathscr{T}^{\prime}\in\mathscr{T}$  for any $\mathscr{T}^{\prime}\subseteq\mathscr{T}$. Each element of $\mathscr{T}$ is called an {\it open set} of the pre-topology.
\end{defi}

Clearly, each topological space is a pre-topological space. Now we provide some examples of pre-topological space such that they are not topological spaces.

\begin{ex}\label{ex1}
(1) Let $Z$ be a set with the cardinality $|Z|\geq 2$. Put $$\mathscr{T}=\{\emptyset\}\cup\{U: U\subseteq Z, |U|\geq 2\}.$$

(2) Let $Z$ be a set with the cardinality $|Z|=\omega$. Put $$\mathscr{T}=\{\emptyset\}\cup\{U: U\subseteq Z, |U|=\omega\}.$$

(3) Let $Z=(-\infty, 1]$ and put $$\mathscr{T}=\{\emptyset\}\cup\{(-\infty, \frac{1}{n}]: n\in\mathbb{N}\}\cup\{[0, \frac{1}{n}]: n\in\mathbb{N}\}\cup\{(-\infty, 0]\}.$$

(4) Let $Z=\{x, y, s, t\}$. Put $\mathscr{T}=\{\emptyset, \{x, y\}, \{x, s\}, \{x, y, s\}, Z\}$.

Clearly, $(Z, \mathscr{T})$ is a pre-topological space in (1)-(4) respectively which is not a topological space.
\end{ex}

\begin{defi}
Let $\mathscr{T}$ and $\mathscr{G}$ be two pre-topologies on a set $X$. We say that the pre-topology $\mathscr{T}$ is coarser than the pre-topology $\mathscr{G}$ when $\mathscr{T}\subseteq\mathscr{G}$; we say that $\mathscr{T}$ is strictly coarser than $\mathscr{G}$ when $\mathscr{T}\subseteq \mathscr{G}$ and $\mathscr{T}\neq \mathscr{G}$. We also say that $\mathscr{G}$ is finer or strictly finer than $\mathscr{T}$ respectively.
\end{defi}

The following theorem shows that any relation $\mathcal{R}$ on $2^{Z}\setminus\{\emptyset\}$ for a nonempty set $Z$ specifies a pre-topology on $Z$.

\begin{thm}
Assume that $Z\neq\emptyset$, and assume that $\mathcal{R}$ is a relation on $2^{Z}\setminus\{\emptyset\}$. If the family $\tau$ of all subsets of $Z$ satisfies the following condition:$$U\in\tau\Leftrightarrow(\forall(K, H)\in\mathcal{R}: K\cap U=\emptyset\Rightarrow H\cap U=\emptyset),$$then $\tau$ is a pre-topology on the set $Z$.
\end{thm}

\begin{proof}
Clearly, $\{\emptyset, Z\}\subseteq\tau$. Let $\mathcal{F}\subseteq\tau$. For each $(K, H)\in\mathcal{R}$, if $K\cap (\bigcup\mathcal{F})=\emptyset$, it follows that $K\cap U=\emptyset$ for any $U\in\mathcal{F}$, hence $H\cap U=\emptyset$ for any $U\in\mathcal{F}$ by our assumption, that is, $H\cap (\bigcup\mathcal{F})=\emptyset$. Therefore, $Z$ is a pre-topological space.
\end{proof}

\maketitle
\subsection{The pre-base and pre-continuous of pre-topologies}
\
\newline
\indent The following definition of pre-base for pre-topological spaces, which has an important role in our discussion of pre-topological spaces.

\begin{defi}
If $\mathcal{B}$ is a family of non-empty subsets of the set $Z$, then $\mathcal{B}$ is called a {\it pre-base} on $Z$ if $\bigcup\mathcal{B}=Z$.
\end{defi}

\begin{rem}
(1) Clearly, each pre-topological space has a pre-base.

(2) In Example~\ref{ex1}, the families $\{W: W\subset Z, |W|=2\}$, $\{W\subset Z: |W|=\omega\}$, $\{[0, \frac{1}{n}]: n\in\mathbb{N}\}\cup\{(-\infty, \frac{1}{n}]: n\in\mathbb{N}\}$ and $\{\{x, y\}, \{x, s\}, Z\}$ are the pre-bases for the pre-topologies of (1), (2), (3) and (4) respectively.
\end{rem}

The following three propositions are obvious.

\begin{prop}
Let $\mathscr{B}$ be a pre-base on the set $Z$. Put $$\tau=\{\bigcup\mathscr{B}^{\prime}: \mathscr{B}^{\prime}\subseteq\mathscr{B}\}.$$ Then $\tau$ is a pre-topology on $Z$ and $\mathscr{B}$ is a {\it pre-base for $\tau$ on $Z$}.
\end{prop}

\begin{prop}\label{p0}
If $(Z, \mathscr{F})$ is a pre-topological space, then, for any $\mathscr{B}\subseteq\mathscr{F}$, the family $\mathscr{B}$ is a pre-base for $\mathscr{F}$ iff for every $W\in\mathscr{F}$ with $z\in W$ there is a $H\in\mathscr{B}$ such that $z\in H\subseteq W$.
\end{prop}

\begin{prop}
Suppose that $\tau$ and $\delta$ are two pre-topologies on $Z$. If $\mathscr{B}$ and $\mathscr{D}$ are pre-bases of $\tau$ and $\delta$ on the set $Z$ respectively, then the following (a) and (b) are equivalent:

\smallskip
(a) $\tau$ is coarser than $\delta$, that is, $\tau\subseteq \delta$;

\smallskip
(b) If $B\in\mathscr{B}$ with $z\in B$, then there is $D\in\mathscr{D}$ with $z\in D\subseteq B$.
\end{prop}

\begin{defi}
If $\mathcal{P}$ is a pre-base of a pre-topological space $Z$, then we say that $\mathscr{P}$ is an {\it atom pre-base} on the set $Z$ if for each $z\in Z$ and $z\in B\in\mathscr{P}$ there does not exist $P\in\mathscr{P}\setminus\{B\}$  such that $z\in P\subseteq B$.
\end{defi}

\begin{rem}\label{rem1}
(1) Any finite pre-topological space has an atom pre-base.

(2) Any Alexandroff space has an atom pre-base, see Section 3.

(3) There exists a pre-topological space which has no atom pre-base.

Indeed, let $Z=[0, 1]$, and put $$\delta=\{\emptyset\}\cup\{[0, \frac{1}{n}]: n\in\mathbb{N}\}.$$Then $(Z, \delta)$ is a pre-topological space which has no atom pre-base.
\end{rem}

\begin{defi}
If $\mathscr{P}$ is a pre-base of a pre-topological space $Z$, then we say that $\mathscr{P}$ is a {\it minimal pre-base} on $Z$ if for any proper subfamily $\mathscr{B}$ of $\mathscr{P}$ is not a pre-base of $Z$.
\end{defi}

Clearly, the following theorem holds.

\begin{thm}
Each atom pre-base of a pre-topological space $Z$ is minimal. In particular, if $Z$ is finite, then each minimal pre-base of $Z$ is an atom pre-base.
\end{thm}

\begin{rem}
The pre-topological space (3) in Remark~\ref{rem1} has a minimal pre-base $\{[0, \frac{1}{n}]: n\in\mathbb{N}\}$, but it has no atom pre-bases.
\end{rem}

\begin{prop}\label{p6}
If $\mathscr{B}$ is a minimal pre-base for a pre-topological space $Z$, then $\mathscr{B}\subseteq \mathscr{F}$ for each pre-base $\mathscr{F}$ of $Z$, which implies that a pre-topological space admits at most one minimal base.
\end{prop}

\begin{proof}
Take an arbitrary pre-base $\mathscr{F}$ of $Z$. Suppose that $B\in\mathscr{B}\setminus\mathscr{F}$. Hence there exists $\mathscr{F}^{\prime}\subseteq \mathscr{F}$ with $B=\bigcup\mathscr{F}^{\prime}$. Then, because $\mathscr{B}$ is a pre-base, it follows that $B$ is a union of a subfamily of $\mathscr{B}\setminus\{B\}$, a contradiction. Therefore, $Z$ admits at most one minimal base.
\end{proof}

The following example shows that some pre-topological spaces have no minimal pre-base.

\begin{ex}
Let $(\mathbb{R}, \mathscr{O})$ be the usual topology. Put $$\mathscr{B}_{1}=\{(a, b): a, b\in \mathbb{Q}\}$$ and $$\mathscr{B}_{2}=\{(a, b): a, b\in \mathbb{P}\}.$$Obviously, the families $\mathscr{B}_{1}$ and $\mathscr{B}_{2}$ are pre-bases of $(\mathbb{R}, \mathscr{O})$ respectively. If $(\mathbb{R}, \mathscr{O})$ has a minimal pre-base $\mathscr{B}$, then it follows from Proposition~\ref{p6} that $\mathscr{B}\subseteq \mathscr{B}_{1}\cap \mathscr{B}_{2}=\emptyset$. Thus, $(\mathbb{R}, \mathscr{O})$ has no minimal pre-base.
\end{ex}

Finally, we give some applications of pre-bases in pre-topology.

\begin{defi}\cite{lijinjin2007}
Let $h: Y\rightarrow Z$ be a mapping between two pre-topological spaces $(Y, \tau)$ and $(Z, \upsilon)$. The mapping $h$ is {\it pre-continuous} from $Y$ to $Z$ if $h^{-1}(W)\in \tau$ for each $W\in\upsilon$.
\end{defi}

Clearly, the following theorem holds.

\begin{thm}\label{t5}
Let $G$, $H$ and $L$ be pre-topological spaces. Then
\begin{enumerate}
\smallskip
\item The identical mapping $i_{L}: L\rightarrow L$ is pre-continuous;

\smallskip
\item If both $h: G\rightarrow H$ and $r: H\rightarrow L$ are pre-continuous, then $r\circ h: G\rightarrow L$ is pre-continuous.
\end{enumerate}
\end{thm}

\begin{defi}\label{d0}
Let $h: Y\rightarrow Z$ be a mapping and $y\in Y$, where $(Y, \tau)$ and $(Z, \upsilon)$ are two pre-topological spaces. The mapping $h$ is {\it pre-continuous} at $z$ if $h^{-1}(W)\in\tau$ with $z\in h^{-1}(W)$ for any $W\in\upsilon$ with $h(z)\in W$.
\end{defi}

By Definition~\ref{d0}, the following two propositions hold.

\begin{prop}
Assume that $h: Y\rightarrow Z$ is a mapping between two pre-topological spaces $(Y, \tau)$ and $(Z, \upsilon)$. For any $y\in Y$, $h$ is pre-continuous at $y$ iff for any $U\in\upsilon$ with $h(y)\in U$ there exists $V\in\tau$ such that $h(V)\subseteq U$ and $y\in V$.
\end{prop}

\begin{prop}
If $r: Y\rightarrow Z$ is a mapping between two pre-topological spaces $Y$ and $Z$, then $r$ is pre-continuous iff $r$ is pre-continuous at each point of $Y$.
\end{prop}

\begin{thm}\label{t1}
If $r: Y\rightarrow Z$ is a mapping between two pre-topological spaces $Y$ and $Z$, then the following are equivalent:
\begin{enumerate}
\smallskip
\item the mapping $r$ is pre-continuous;

\smallskip
\item there exists a pre-base $\mathscr{B}$ in $Z$ with $r^{-1}(B)$ is open in $Y$ for each $B\in\mathscr{B}$.
\end{enumerate}
\end{thm}

For each pre-topological space $Z$ and $z\in Z$, denote the set all open neighborhoods of $z$ by $\mathscr{U}_{z}$.

\begin{defi}
Let $Z$ be a pre-topological space and $z\in Z$. The family $\mathscr{V}_{z}\subseteq\mathscr{U}_{z}$ is an {\it open neighborhood pre-base} at $z$ if for each $U\in\mathscr{U}_{z}$ there exists $V\in\mathscr{V}_{z}$ with $V\subseteq U$.
\end{defi}

\begin{thm}
Let $h: Y\rightarrow Z$ be a mapping between two pre-topological spaces $(Y, \tau)$ and $(Z, \upsilon)$. For any $y\in Y$, the following (1) $\Leftrightarrow$ (2).
\begin{enumerate}
\smallskip
\item $h$ is pre-continuous at $y\in Y$.

\smallskip
\item there exists a pre-base $\mathscr{B}_{h(y)}\subseteq \upsilon$ at $h(y)$ such that $h^{-1}(B)\in\tau$ with $y\in h^{-1}(B)$ for each $B\in\mathscr{B}_{h(y)}$.
\end{enumerate}
\end{thm}

\begin{defi}\label{d1}
Let $h: Y\rightarrow Z$ be a bijection between two pre-topological spaces $Y$ and $Z$. If $h$ and $h^{-1}: Z\rightarrow Y$ are all pre-continuous, then we say that $h$ is a {\it pre-homeomorphic mapping}. We also say that $Y$ and $Z$ are {\it pre-homeomorphic}.
\end{defi}

\begin{thm}
If $G$, $H$ and $Z$ be pre-topological spaces, then
\begin{enumerate}
\smallskip
\item the identical mapping $i_{Z}: Z\rightarrow Z$ is a pre-homeomorphic mapping;

\smallskip
\item if $h: G\rightarrow H$ is a pre-homeomorphic mapping, then $h^{-1}: H\rightarrow G$ is pre-homeomorphic;

\smallskip
\item if both $h: G\rightarrow H$ and $r: H\rightarrow Z$ are pre-homeomorphic, then $r\circ h: G\rightarrow Z$ is pre-homeomorphic.
\end{enumerate}
\end{thm}

Note that the pre-homeomorphism $h: Y\rightarrow Z$ gives us a bijection respondence between the collections of open sets of $Y$ and of $Z$. Then a property $\mathcal{P}$ of $Y$ is a {\it pre-topological property} if $Y$ has the property $\mathcal{P}$ then any pre-homeomorphism image of $Y$ has the the property $\mathcal{P}$.

Assume that $h: Y\rightarrow Z$ is a pre-continuous injective mapping between two pre-topological spaces $Y$ and $Z$. Then the mapping $h$ is a {\it pre-topological imbedding} of $Y$ in $Z$ if $h: Y\rightarrow h(Y)$ is a pre-homeomorphism.

\maketitle
\subsection{Subsets of pre-topology}
\
\newline
\indent In this subsection, we discuss subsets of a pre-topological space, and treat the notions of closed set, closure of a set, interior of a set, boundary of a set, and accumulation point. The reader can see the proofs of some results in \cite{lijinjin2006}.

\begin{defi}
A subset $D$ of a pre-topological space $Z$ is {\it closed} provided $Z\setminus D$ is open in $Z$.
\end{defi}

The following proposition is easily verified.

\begin{prop}\label{p1}
Let $(Z, \tau)$ be a pre-topological space, and put $\mathcal{C}=\{F: Z\setminus F\in\tau\}$; then
\begin{enumerate}
\smallskip
\item $\emptyset, Z\in\mathcal{C}$;

\smallskip
\item for any non-empty $\mathcal{C}_{0}\subseteq\mathcal{C}$, $\bigcap_{C\in\mathcal{C}_{0}}C\in\mathcal{C}$.
\end{enumerate}
Conversely, if a subfamily $\mathcal{C}\subseteq 2^{Z}$ satisfies (1) and (2) above, then
$$\tau=\{U\subseteq Z: Z\setminus U\in\mathcal{C}\}$$ is a pre-topology on $Z$.
\end{prop}

Take an arbitrary subset $F$ of a pre-topological space $(Z, \tau)$;  then it follows from Proposition~\ref{p1} that $$\bigcap\{C: F\subseteq C, Z\setminus C\tau\}$$ is closed in $Z$, which is called the {\it closure} of $F$ and denoted by $\overline{F}$. Clearly, $\overline{F}$ is the smallest closed set containing $F$, and a set $C$ is closed iff $C=\overline{C}$. The proof of the following proposition left to the reader.

\begin{prop}
Let $H$ be a pre-topological space and $Y\subseteq H$. Then
\begin{enumerate}
\smallskip
\item $\overline{\emptyset}=\emptyset$;

\smallskip
\item $Y\subseteq \overline{Y}$;

\smallskip
\item $\overline{Y}=\overline{\overline{Y}}$.
\end{enumerate}
\end{prop}

For a topological space $Z$, it is well known that for any subsets $R, T\subseteq Z$ we have $\overline{R\cup T}=\overline{R}\cup\overline{T}$. However, the equality of $\overline{R\cup T}=\overline{R}\cup\overline{T}$ does not hold for pre-topological spaces. Indeed, the following example clarifies this point.

\begin{ex}\label{e0}
Let $Z=\{z_{1}, z_{2}, z_{3}, z_{4}\}$ and $$\tau=\{\emptyset, \{z_{1}, z_{2}\}, \{z_{1}, z_{4}\}, \{z_{1}, z_{3}\}, \{z_{1}, z_{2}, z_{3}\}, \{z_{1}, z_{3}, z_{4}\}, \{z_{1}, z_{2}, z_{4}\}, Z\}.$$ Then $(Z, \tau)$ is a pre-topology on $Z$. Let $R=\{z_{2}, z_{3}\}$ and $T=\{z_{3}, z_{4}\}$. Clearly, $\overline{R}=R$, $\overline{T}=T$ and $\overline{R\cup T}=\overline{\{z_{2}, z_{3}, z_{4}\}}=Z$, thus $\overline{R\cup T}\neq \overline{R}\cup \overline{T}.$
\end{ex}

\begin{prop}\label{p2}
If $Y$ is a subset of a pre-topological space $(Z, \tau)$ and $z\in Z$, then $z\in\overline{Y}$ iff $U\cap Y\neq\emptyset$ for any $U\in\tau$ with $z\in U$.
\end{prop}

\begin{cor}
Let $U$ be open in a pre-topological space $Z$ and $A\subseteq Z$. If $U\cap A=\emptyset$, then $U\cap \overline{A}=\emptyset$.
\end{cor}

\begin{defi}
Let $Z$ be a pre-topological space $(Z, \tau)$, $C\subseteq Z$ and $z\in Z$. We say that $z$ is an {\it accumulation point} of $C$ if $U\cap (C\setminus\{z\})\neq\emptyset$ for any $U\in\tau$ with $z\in U$. The {\it derived set} of $C$ is the set of all accumulation points of $C$ and denoted by $C^{d}$.
\end{defi}

\begin{prop}\label{p5}
If $B$ is a subset of a pre-topological space $Z$, then $\overline{B}=B\cup B^{d}$.
\end{prop}

\begin{cor}
If $G$ is a subset of a pre-topological space $Z$, then $G^{d}\subseteq G$ iff $G$ is closed in $Z$.
\end{cor}

Some equivalent ways of formulating the concept of pre-continuity are provided as follows.

\begin{thm}
Let $h: Y\rightarrow Z$ be a mapping between two pre-topological spaces $Y$ and $Z$. Then the following are equivalent:

\smallskip
(i) $h$ is pre-continuous;

\smallskip
(ii) $h(\overline{B})\subseteq \overline{h(B)}$ for each $B\subseteq Y$;

\smallskip
(iii) $h^{-1}(H)$ is closed in $Y$ for any closed set $H$ in $Z$.
\end{thm}

A {\it closure operator} $c$ on $Z$ is an operator that assigns to each subset $B$ of $Z$ a subset $B^{c}$ of $Z$ so that the following four statements hold.

\smallskip
(a) $\emptyset^{c}=\emptyset$.

\smallskip
(b) $B\subseteq B^{c}$ for each $B\subset Z$.

\smallskip
(c) $B^{cc}= B^{c}$ for each $B\subset Z$.

\smallskip
(d) If $B\subseteq D$, then $B^{c}\subseteq D^{c}$.

The following theorem shows that these four statements are actually characteristic of closure. The pre-topology defined below is the pre-topology associated with a closure operator.

\begin{thm}
Let $c$ be a closure operator on $Z$. If $$\mathcal{F}=\{A: A\subseteq Z, A^{c}=A\}$$ and $$\tau=\{U: Z\setminus U\in \mathcal{F}\},$$ then $\tau$ is a pre-topology on $Z$ and $A^{c}$ is the closure of $A$ for every $A\subseteq Z$.
\end{thm}

\begin{defi}
If $B$ is a subset of a pre-topological space $(Z, \tau)$, then the set $$\bigcup\{W\subseteq B: W\in\tau\}$$ is called the {\it interior} of $B$ and is denoted by $\mbox{int}B$ or $B^{\circ}$.
\end{defi}

Clearly, the following proposition holds.

\begin{prop}
If $G$ and $H$ are subsets of a pre-topological space $(Z, \tau)$, then the following (1)-(3) hold.
\begin{enumerate}
\smallskip
\item $G\subseteq H\Rightarrow G^{\circ}\subseteq H^{\circ}$.

\smallskip
\item $G\in\tau$ $\Leftrightarrow$ $G=G^{\circ}$.

\smallskip
\item $G$ is closed in $Z$ $\Leftrightarrow$ $Z\setminus G=(Z\setminus G)^{\circ}$.
\end{enumerate}
\end{prop}

\begin{thm}\label{t2}
If $G$ is a subset of a pre-topological space $H$, then
\begin{enumerate}
\smallskip
\item $G^{\circ}=H\setminus\overline{H\setminus G}$;

\smallskip
\item $\overline{G}=H\setminus (H\setminus G)^{\circ}$.
\end{enumerate}
\end{thm}

\begin{defi}\label{d2}
Let $Z$ be a pre-topological space and $B\subseteq Z$. The set $\partial B=\overline{B}\cap \overline{Z\setminus B}$ is said to be the {\it boundary} of $B$.
\end{defi}

Clearly, $\partial B=\partial (Z\setminus B)$ is a closed set. From Proposition~\ref{p2} it follows that $z\in\partial B$ iff, for each open neighborhood $O$ of $z$, we have $O\cap (Z\setminus B)\neq\emptyset$ and $O\cap B\neq\emptyset$.

\begin{thm}\label{t3}
If $G$ is any subset of pre-topological space $(Z, \tau)$, then the following equalities hold.
\begin{enumerate}
\smallskip
\item $\partial G=\overline{G}\setminus G^{\circ}$;

\smallskip
\item $G^{\circ}=G\setminus\partial G$.
\end{enumerate}
\end{thm}

\begin{rem}
(1) It follows from Theorem~\ref{t3} that

(a) $G\in\tau$ $\Leftrightarrow$ $\partial G=\overline{G}\setminus G$;

(b) $Z\setminus G\in\tau$ $\Leftrightarrow$ $\partial G=G\setminus G^{\circ}$;

(c) $G$ is open and closed $\Leftrightarrow$ $\partial G=\emptyset.$

(2) For a topological space $Z$, we have $\partial (R\cup T)\subseteq\partial R\cup \partial T$. However, this equality does not hold for pre-topological spaces. Indeed, let $Z$ be the pre-topological space in Example~\ref{e0}, and let $R$ and $T$ be the subsets of Example~\ref{e0}. Then $\partial R=R$, $\partial T=T$ and $\partial (R\cup T)=Z$, thus $\partial (R\cup T)\nsubseteq\partial R\cup \partial T$.
\end{rem}

\maketitle
\subsection{subspaces of pre-topological spaces}
\
\newline
\indent Let $Y$ be a subset of pre-topological space $(Z, \tau)$. Put $$\tau|_{Y}=\{O\cap Z: O\in\tau\}.$$ Then $(Y, \tau|_{Y})$ is a pre-topology too.  We say that $(Y, \tau|_{Y})$ is the {\it subspace} of $(Z, \tau)$.

First, the following proposition and theorem hold.

\begin{prop}
If $Y$ is a subspace of pre-topological space $Z$, then the family $\mathscr{B}|_{Y}=\{O\cap Z: O\in\mathscr{P}\}$ is a pre-base for $Y$, where $\mathscr{P}$ is a pre-base for $Z$.
\end{prop}

\begin{thm}\label{t4}
Let $H\subseteq Z$ and $F\subseteq H$, where $(Z, \tau)$ is a pre-topological space. Then
\begin{enumerate}
\smallskip
\item $F$ is closed in $H$ $\Leftrightarrow$ there is a closed subset $F_{1}$ in $Z$ so that $F=F_{1}\cap H$;

\smallskip
\item $\mbox{cl}_{H}F=H\cap \overline{F}$.
\end{enumerate}
\end{thm}

\begin{cor}\label{c0}
If $H$ is closed in the pre-topological space $Z$ and $A\subseteq H$, then $A$ is closed in $H$ iff $A$ is closed in $Z$.
\end{cor}

\begin{rem}
(1) Let $Z$ be the pre-topological space in Example~\ref{e0}. Then $W=\{z_{1}, z_{2}\}$ and $\{z_{1}\}$ are open in $Z$ and $W$ respectively. However, $\{z_{1}\}$ is not open in $Z$.

(2) It is well-known that `Pasting Lemma' and `Local formulation of continuity' play important roles in the study of the continuous theory of general topology. However, `Pasting Lemma' and `Local formulation of continuity' do not hold in the class of pre-topological spaces, see the following Example~\ref{e1}.
\end{rem}

\begin{ex}\label{e1}
Let $Z=\{z_{1}, z_{2}, z_{3}, z_{4}\}$, and let $$\tau=\{\emptyset,\{z_{1}, z_{2}\}, \{z_{1}, z_{4}\}, \{z_{1}, z_{3}\}, \{z_{2}, z_{3}\}, \{z_{2}, z_{4}\}, \{z_{3}, z_{4}\}, \{z_{1}, z_{2}, z_{4}\}, \{z_{1}, z_{2}, z_{3}\}, $$$$\{z_{1}, z{3}, z_{4}\}, \{z_{2}, z_{3}, z_{4}\}, Z\}$$
 and $$\delta=\{\emptyset, \{z_{3}\}, \{z_{1}, z_{2}\}, \{z_{1}, z_{3}\}, \{z_{1}, z_{4}\},  \{z_{1}, z_{2}, z_{3}\}, \{z_{1}, z_{3}, z_{4}\}, \{z_{1}, z_{2}, z_{4}\}, Z\}.$$
Then $(Z, \tau)$ and $(Z, \delta)$ are two pre-topologies on $Z$ respectively. Let $C=\{z_{1}, z_{2}\}$ and $D=\{z{3}, z_{4}\}$. Then
$\tau|_{C}=\{\emptyset, \{z_{1}\}, \{z_{2}\}, C\}$ and $\tau|_{D}=\{\emptyset, \{z_{3}\}, \{z_{4}\}, D\}$. Clearly, $Z=C\cup D$ and $C$ and $D$ are clopen in $(Z, \tau)$ . Let $h: (Z, \tau)\rightarrow (Z, \delta)$ be the identical mapping from $Z$ to itself. Then $$h|_{C}: (C, \tau|_{C})\rightarrow (Z, \delta)\ \mbox{and}\ h|_{D}:(D, \tau|_{G})\rightarrow (Z, \delta)$$ are pre-continuous.
However, $h$ is not pre-continuous. Indeed, the set $\{z_{3}\}$ is an open set in $(Z, \delta)$; however, the set $h^{-1}(\{z_{3}\})=\{z_{3}\}$ is not open in $(Z, \tau)$.
\end{ex}

\begin{prob}\label{p3}
Assume that $h: G\rightarrow H$ is a mapping between two pre-topological spaces $G$ and $H$ such that there exist closed subsets $C$ and $D$ in $G$ satisfying the following conditions hold.
\begin{enumerate}
\smallskip
\item $G=C\cup D$;

\smallskip
\item the restrict mappings $h|_{C}$ and $h|_{D}$ are pre-continuous.
\end{enumerate}
Under what conditions must $h$ be pre-continuous?
\end{prob}

By Corollary~\ref{c0}, the following theorem gives a partial answer to Problem~\ref{p3}.

\begin{thm}
Assume that $h: G\rightarrow H$ is a mapping between two pre-topological spaces $G$ and $H$ such that there exist closed subsets $C$ and $D$ in $G$ satisfying the following conditions hold.
\begin{enumerate}
\smallskip
\item $G=C\cup D$;

\smallskip
\item $h^{-1}(F)\subseteq C$ or $f^{-1}(F)\subseteq D$ for each closed subset $F$ of $H$;

\smallskip
\item $h|_{C}$ and $h|_{D}$ are pre-continuous.
\end{enumerate}
Then $h$ is pre-continuous.
\end{thm}

The following proposition is obvious.

\begin{prop}
Let $Y$ and $Z$ be pre-topological spaces.
\begin{enumerate}
\smallskip
\item The inclusion mapping $j: B\rightarrow Z$ is pre-continuous, where $B$ is a subspace of $Z$.

\smallskip
\item If $r: Y\rightarrow Z$ is pre-continuous, then the restricted mapping $r|_{H}$ is pre-continuous for any subspace $H$ of $Y$.
\end{enumerate}
\end{prop}

Let $B$ be a subspace of a pre-topological space $(Z, \tau)$. For any $U\in\tau$, put $[U]_{B}=\{O\in\tau: O\cap B=U\cap B\}$; then $[U]_{B}$ is the set of all elements in $\tau$ such that each element has the same trace on $B$; moreover, $W\cap B\subseteq\bigcap[U]_{B}$ for each $W\in [U]_{B}$.

\begin{thm}
Let $B$ be a subspace of pre-topological space $(Z, \tau)$. For each $U\in \tau$, the set $\gamma=\{W: W=T\setminus(\bigcap[U]_{B}), T\in[U]_{B}\}\cup\{\emptyset\}$ is a pre-topology on $\bigcup\gamma$.
\end{thm}

\begin{proof}
Clearly, it suffices to prove that $\bigcup\eta\in\gamma$ for any $\eta\subseteq\gamma$. Take any $\eta\subseteq\gamma$. Then, for each $W\in\eta$, there is $V_{W}\in[U]_{B}$ with $W=V_{W}\setminus(\bigcap[U]_{B})$. Put $\eta=\{V_{W}: W\in\eta\}$. Then
$\bigcup\eta=(\bigcup\eta)\setminus(\bigcap[U]_{B})$ and $B\cap(\bigcup\eta)=V\cap B$ for any $V\in \eta$. Since $Z$ is a pre-topological space, it follows that $\bigcup\eta$ is open in $Z$, hence $\bigcup\eta\in [U]_{B}$, which implies that $\bigcup\eta=(\bigcup\eta)\setminus(\bigcap[U]_{B})\in\gamma$.
\end{proof}

We say that $(\bigcup\gamma, \gamma)$ is an {\it $B$-child pre-topological space of $\tau$}, or simply {\it child of $\tau$}. Clearly, $\gamma$ is coarser than the pre-topological subspace $B$ of $Z$.

\maketitle
\subsection{The product pre-topology}
\
\newline
\indent For pre-topological spaces $Y$ and $Z$, there is a standard way to define a pre-topology on the cartesian product $Y\times Z$.

\begin{defi}
Let $(Y, \tau)$ and $(Z, \eta)$ be pre-topological spaces, and let $\mathscr{B}=\{V\times W: V\in\tau, W\in\eta\}$. The {\it product pre-topology} on $Y\times Z$ is the pre-topology on the set $Y\times Z$ so that $\mathscr{B}$ is a pre-basis on $Y\times Z$.
\end{defi}

\begin{defi}
Let $\{Z_{\beta}\}_{\beta\in J}$ be a family of pre-topological spaces. Let $$\mathcal{B}=\{\prod_{\beta\in J}W_{\beta}: W_{\beta}\ \mbox{is open in}\ Z_{\beta}, \beta\in J\}.$$We say a pre-topology $\tau$ on the product space $\prod_{\beta\in J}Z_{\beta}$ is called the {\it box pre-topology} if $\mathcal{B}$ is a pre-basis of $\tau$.
\end{defi}

Suppose that $\{Z_{\beta}\}_{\beta\in J}$ is an indexed of family of sets. For any subset $I\subseteq J$, let $$\pi_{I}: \prod_{\alpha\in J}X_{\alpha}\rightarrow \prod_{\beta\in I}Z_{\beta}$$ be the mapping which assigns to every element of the product spaces its $\alpha$th coordinate for any $\alpha\in I$, $$\pi_{I}((z_{\beta})_{\beta\in J})=(z_{\beta})_{\beta\in I}$$it is said to be the {\it projection mapping} associated with any index $\alpha\in I$.

\begin{defi}
Let $\{(X_{\alpha}, \tau_{\alpha})\}_{\alpha\in J}$ be a family of pre-topological spaces. For any finite subset $I\subseteq J$, let $\mathcal{S}_{I}$ denote the collection
$$\mathcal{S}_{I}=\{\pi_{I}^{-1}(\prod_{\alpha\in F}U_{\alpha}): U_{\alpha}\in\tau_{\alpha}, \alpha\in I\},$$and let $$\mathcal{S}=\bigcup\{\mathcal{S}_{I}: I\ \mbox{is any finite subset of}\ J\}.$$The pre-topology is said to be the {\it product pre-topology} if it is generated by the collection $\mathcal{S}$. In this pre-topology $\prod_{\alpha\in J}X_{\alpha}$ is called a {\it product pre-topological space}.
\end{defi}

The following three theorems are important.

\begin{thm}
Let $B_{\beta}$ be a subspace of $Z_{\beta}$ for each $\beta\in J$. Assume that both products $\prod B_{\beta}$ and $\prod Z_{\beta}$ are endowed with the box pre-topology or both products are endowed with the product pre-topology, then $\prod B_{\beta}$ is a subspace of $\prod Z_{\beta}$.
\end{thm}

\begin{thm}
Assume that $\{X_{\alpha}\}$ is a family of pre-topological spaces, and assume that $A_{\alpha}\subseteq X_{\alpha}$ for each $\alpha$. If $\prod X_{\alpha}$ is endowed with either the box or the product pre-topology, then $$\prod\overline{A_{\alpha}}=\overline{\prod A_{\alpha}}.$$
\end{thm}

\begin{thm}\label{t10}
Assume that $h: B\rightarrow\prod_{\beta\in J}Z_{\beta}$ is a mapping which is defined as follows
$$h(b)=(h_{\beta}(b))_{\beta\in J},$$where $b\in B$ and $h_{\beta}: B\rightarrow Z_{\beta}$ for every $\beta\in J$. If $\prod_{\beta\in J}Z_{\beta}$ has the product pre-topology, then the mapping $h$ is pre-continuous iff each $h_{\beta}$ is pre-continuous.
\end{thm}

\maketitle
\subsection{The quotient pre-topology}
\
\newline
\indent In this subsection, we consider the pre-quotient pre-topology, which is a new way of generating a pre-topology on a set.

\begin{defi}
Let $p: Y\rightarrow Z$ be a surjective mapping between two pre-topological spaces $Y$ and $Z$. The mapping $p$ is called a {\it pre-quotient mapping} provided $W$ is open in $Z$ iff $p^{-1}(W)$ is open in $Y$.
\end{defi}

\begin{defi}
Let $r: Y\rightarrow Z$ be a mapping between two pre-topological spaces $(Y, \tau)$ and $(Z, \eta)$. If $r(W)\in\eta$ for each $W\in\tau$, then the mapping $r$ is called a {\it pre-open mapping}. If the set $r(C)$ is always closed in $Z$ for each closed subset $C$ of $Y$, then the mapping $r$ is called a {\it pre-closed mapping}.
\end{defi}

Clearly, each pre-quotient mapping is a pre-continuous mapping but not vice-versa. Moreover, since $r^{-1}(Z\setminus B)=Y\setminus r^{-1}(B)$, it follows that $h$ is pre-quotient mapping provided $r^{-1}(C)$ is closed in $Y$ iff $C$ is closed in $Z$. Moreover, if $r: Y\rightarrow Z$ is a surjective pre-continuous mapping which is either pre-closed or pre-open, then from the definition it follows that $r$ is a pre-quotient mapping. From Example~\ref{ex2}, it is easily checked that there exists a pre-quotient mapping that is neither pre-open nor pre-closed.

\begin{ex}\label{ex2}
Let $\mathbb{R}$ endowed with the pre-topology $\tau$ which has a pre-base $\{(x, y): |y-x|=1\}$. Suppose that $\pi_{1}: \mathbb{R}\times \mathbb{R}\rightarrow\mathbb{R}$ is the projection on the first coordinate. Let $$G=\{(x, y): y= 0\}\cup\{(x, y): x\geq 0\},$$ and let $p: G\rightarrow \mathbb{R}$ be obtained by restricting $\pi_{1}$. Then it is easy to see that $p$ is a pre-quotient mapping which is neither pre-closed nor pre-open.
\end{ex}

Now we can use the concept of pre-quotient mapping to construct a pre-topology on a set.

\begin{defi}
Let $h: Y\rightarrow Z$ be a surjective mapping, where $Z$ is a set and $Y$ is a pre-topological space. Then the following pre-topology $$\mathcal{F}=\{O\subseteq Z: h^{-1}(O)\ \mbox{is open in}\ Y\}$$ on $Z$ is said to be the {\it pre-quotient pre-topology induced by $h$}.
\end{defi}

\begin{thm}
Let $h: Y\rightarrow Z$ be a surjective mapping, where $Z$ is a set and $(Y, \tau)$ is a pre-topological space. Then the pre-quotient pre-topology on $Z$ is the finest pre-topology on $Z$ so that $h$ is pre-continuous.
\end{thm}

\begin{proof}
Let $\delta=\{W\subseteq Z: h^{-1}(W)\ \mbox{is pre-open in}\ Y\}$ be the pre-quotient pre-topology on $Z$. For each $W\in\delta$, since $h^{-1}(W)\in\tau$, we conclude that $h: (Y, \tau)\rightarrow (Z, \delta)$ is pre-continuous. Assume that there is a pre-topology $\delta^{\prime}$ on $Z$ so that $h: (Y, \tau)\rightarrow (Z, \delta^{\prime})$ is pre-continuous. Since $h$ is pre-continuous, it follows that for each $O\in\delta^{\prime}$ we have $h^{-1}(O)\in\tau$, hence $\delta^{\prime}\subseteq\delta$. Therefore, $\delta$ is the finest pre-topology on $Z$ which makes $h$ pre-continuous.
\end{proof}

Clearly, the following proposition holds.

\begin{prop}\label{p4}
Let $h: G\rightarrow H$ and $g: H\rightarrow Z$ be pre-quotient mapping, where $G$, $H$ and $Z$ are pre-topological spaces. Then $g\circ h: G\rightarrow Z$ is a pre-quotient mapping too.
\end{prop}

From Theorem~\ref{t5}, Proposition~\ref{p4} and the definition of pre-quotient mapping, it is easily checked the proofs of the following two theorems.

\begin{thm}
Suppose that $X_{1}$, $X_{2}$ and $X_{3}$ are pre-topological spaces, and suppose that $h: X_{1}\rightarrow X_{2}$ is pre-quotient and $g: X_{2}\rightarrow X_{3}$ is a mapping. We conclude that the following statements hold:
\begin{enumerate}
\smallskip
\item $g\circ h: X_{1}\rightarrow X_{3}$ is pre-continuous $\Leftrightarrow$ $g$ is pre-continuous.

\smallskip
\item $g\circ h: X_{1}\rightarrow X_{3}$ is pre-quotient $\Leftrightarrow$ $g$ is pre-quotient.
\end{enumerate}
\end{thm}

\begin{thm}
Let $Z_{1}$ and $Z_{2}$ be pre-topological spaces, and let $h: Z_{1}\rightarrow Z_{2}$ be pre-continuous and bijective. Then the following (1)-(4) are equivalent:
\begin{enumerate}
\smallskip
\item $h$ is pre-homeomorphic;

\smallskip
\item $h$ is pre-closed;

\item $h$ is pre-open;

\smallskip
\item $h$ is pre-quotient.
\end{enumerate}
\end{thm}

\begin{defi}
Let $E$ be an equivalence relation on the set $Z$, and let $\tau$ be a pre-topology on $Z$. Let $Z/E$ denote the set of all equivalence classes of $E$ and let $q$ denote the mapping of $Z$ to $Z/E$ assigning to the point $z$ the equivalence class $[z]\in Z/E$. Then the set $Z/E$ endowed with the pre-quotient pre-topology on is called {\it pre-quotient space}.
\end{defi}

Assume $(Z, \tau)$ is a pre-topological space. It can define an equivalence relation $R$ on the set $Z$ by $x\thicksim y\Leftrightarrow \tau_{x}=\tau_{y}$. Put $Z^{\ast}=\{[z]: z\in Z\}$, and let $W^{\ast}=\{[z]: z\in W\}$ for each $W\in\tau$. Put $\tau^{\ast}=\{W^{\ast}: W\in \tau\}$. Then $\tau^{\ast}$ is the pre-quotient pre-topology induced from $\tau$ on $Z^{\ast}$. Moreover, $(Z^{\ast}, \tau^{\ast})$ is $T_{0}$ (see Definition~\ref{d3} below). The pre-topological space $(Z^{\ast}, \tau^{\ast})$ is called the {\it $T_{0}$ reduction of $(Z, \tau)$}.

\maketitle
\subsection{Axioms of separation}
\
\newline
\indent In this subsection, we consider the axioms of separation in the class of pre-topological spaces, which give the ways of separating points and closed sets in the study of pre-topological spaces.

\begin{defi}\label{d3}
A pre-topological space $(Z, \tau)$ is called a {\it $T_{0}$-space} if for any $y, z\in Z$ with $y\neq z$ there exists $W\in\tau$ such that $W\cap \{y, z\}$ is exact one-point set.
\end{defi}

\begin{defi}
A pre-topological space $(Z, \tau)$ is called a {\it $T_{1}$-space} if for any $y, z\in Z$ with $y\neq z$ there are $V, W\in\tau$ so that $V\cap\{y, z\}=\{y\}$ and $W\cap\{y, z\}=\{z\}$.
\end{defi}

\begin{defi}
A pre-topological space $(Z, \tau)$ is called a {\it $T_{2}$-space}, or a {\it Hausdorff space}, if for any $y, z\in Z$ with $y\neq z$ there are $V, W\in\tau$ so that $y\in V$, $z\in W$ and $V\cap W=\emptyset$.
\end{defi}

\begin{rem}\label{rem2}
(1) Suppose that $Z=\{z_{1}, z_{2}, z_{3}, z_{4}, z_{5}, z_{6}\}$ equipped with the pre-topology as follows:
$$\tau=\{\emptyset, \{z_{4}\}, \{z_{5}, z_{6}\}, \{z_{1}, z_{3}\}, \{z_{1}, z_{3}, z_{4}\}, \{z_{4}, z_{5}, z_{6}\}, \{z_{1}, z_{2}, z_{3}\}, \{z_{1}, z_{2}, z_{3}, z_{4}\}, $$$$\{z_{1}, z_{3}, z_{5}, z_{6}\}, \{z_{1}, z_{3}, z_{4}, z_{5}, z_{6}\}, Z\}.$$Hence $(Z, \tau)$ is not a $T_{0}$ pre-topological space.

(2) Example~\ref{e0} is a $T_{0}$ pre-topological space.

(3) Let $Z$ be an infinite set equipped with the following pre-topology
$$\tau=\{\emptyset\}\cup\{U: |Z\setminus U|\leq 2\}.$$ Then the pre-topological space $(Z, \tau)$ is $T_{1}$.

(4) In Example~\ref{e1}, the pre-topological space $(Z, \tau)$ is $T_{2}$.

(5) Each $T_{1}$ pre-topological space is $T_{0}$, and each $T_{2}$ pre-topological space is $T_{1}$; but all the inverses are not true. Indeed, Example~\ref{e0} is a $T_{0}$ pre-topological space that is not $T_{1}$; $(X, \tau)$ in (3) above is a $T_{1}$ pre-topological space that is not $T_{2}$.
\end{rem}

Each set of cardinal numbers being well-ordered by $<$. Let $Z$ be a pre-topological space. Then the infimum of the set $$\{|\mathcal{B}|: \mathcal{B}\ \mbox{is a pre-base for}\ Z\}$$ is said to be the {\it weight} of $Z$ and is denoted by $w(Z)$.

\begin{thm}\label{t13}
For each $T_{0}$ pre-topological space $(Z, \tau)$ we have $|Z|\leq 2^{w(Z)}$, thus $|Z|\leq |\tau|$.
\end{thm}

\begin{proof}
Let $\mathcal{B}$ be a pre-base for $Z$ such that $|\mathcal{B}|=w(Z)$. For every $z\in Z$, put $\mathcal{B}_{z}=\{W\in\mathcal{B}: z\in W\}$. Because $Z$ is a $T_{0}$ pre-topological space, it follows that $\mathcal{B}_{x}\neq \mathcal{B}_{y}$ for any $x\neq y$. Because the number of all distinct families $\mathcal{B}_{x}$ is not larger than $2^{|\mathcal{B}|}$, we conclude that $|Z|\leq 2^{w(Z)}$.
\end{proof}

\begin{defi}
The {\it locally inner closed points} of an open subset $W$ of a $T_{0}$ pre-topological space $(Z, \tau)$, is the subset of points
$$W^{\mathcal{I}}=\{z\in W: W\setminus\{z\}\in\tau\}.$$

The {\it locally outer closed points} of an open subset $W$ of a $T_{0}$ pre-topological space $(Z, \tau)$, is the subset of points
$$W^{\mathcal{O}}=\{z\in Z\setminus W: W\cup\{z\}\in\tau\}.$$

The {\it locally closed points} of an open subset $W$ of a $T_{0}$ pre-topological space is the set $$W^{\mathcal{\mathcal{LC}}}=W^{\mathcal{I}}\cup W^{\mathcal{O}}.$$
\end{defi}

\begin{prop}
Let $Z$ be a $T_{0}$ pre-topological space, and let $V$ and $W$ be open in $Z$ such that $(V\triangle W)\subseteq V^{\mathcal{LC}}$ or $(V\triangle W)\subseteq W^{\mathcal{\mathcal{LC}}}$. Then $V^{\mathcal{I}}=W^{\mathcal{I}}$ and $V^{\mathcal{O}}=W^{\mathcal{O}}$ iff $V=W$.
\end{prop}

\begin{proof}
Clearly, we only need to prove the necessity. Suppose that $V^{\mathcal{I}}=W^{\mathcal{I}}$ and $V^{\mathcal{O}}=W^{\mathcal{O}}$. Now let $V\neq W$; hence $\emptyset\neq V\triangle W\subseteq V^{\mathcal{\mathcal{LC}}}=W^{\mathcal{LC}}$. Take any $x\in V\triangle W$. Without loss of generality, we assume that $x\in V$ and $x\not\in W$. Clearly, $x\not\in V^{\mathcal{O}}$ because $x\in V$, and $x\not\in V^{\mathcal{I}}$ since $x\not\in W^{\mathcal{I}}$ and $V^{\mathcal{I}}=W^{\mathcal{I}}$, hence $x\not\in V^{\mathcal{\mathcal{LC}}}$, a contradiction.
\end{proof}

\begin{defi}
A subset $D$ of a pre-topological space $Z$ is called {\it dense} in $Z$ provided $\overline{D}=Z$.
\end{defi}

In a $T_{2}$ topological space $Z$ with a dense subset $A$, we have $|X|\leq2^{2^{|A|}}$ and $\overline{W\cap A}=\overline{W}$ for each open set $W$ in $Z$. However, from the following example, we conclude that in the class of $T_{2}$ pre-topological spaces, these situations are different.

\begin{ex}\label{e2}
There exists a $T_{2}$ pre-topological space $(Z, \tau)$ with a dense subset $C$ such that the following are satisfied.
\begin{enumerate}
\smallskip
\item $\overline{W\cap C}\neq\overline{W}$ for some $W\in\tau$;

\smallskip
\item $|Z|>2^{2^{|C|}}$.
\end{enumerate}
\end{ex}

\begin{proof}
Let $Z$ be a set with $|Z|>2^{8}$, and let
$$\mathscr{B}=\{\{x, y\}: x\in\{a, b, c\}, y\in Z, x\neq y\},$$ where $a, b, c$ are three distinct points of $Z$. The pre-topology $\tau$ is endowed on $Z$ such that $\mathcal{B}$ is a pre-base of $\tau$.
Obviously, $(Z, \tau)$ is a $T_{2}$ pre-topological space, and the set $C=\{a, b, c\}$ is a dense subset of $(Z, \tau)$.

(1) Take the open set $W=\{c, d\}$. Then we have $\overline{\{c, d\}\cap C}=\overline{\{c\}}=\{c\}$ and $\overline{\{c, d\}}=\{c, d\}$, which shows $\overline{W\cap C}\neq\overline{W}$.

(2) Since $|Z|>2^{8}$ and $2^{2^{|C|}}=2^{8}$, it follows that $|Z|>2^{2^{|C|}}$.
\end{proof}

Let $Z$ be a pre-topological space. For each $z\in Z$, the infimum of the set $$\{|\mathcal{B}(z)|: \mathcal{B}(z)\ \mbox{is a pre-base at}\ z\}$$
is called the {\it character of a point $z$ } in $Z$, which is denoted by $\chi(z, Z)$. The supremum of the set $\{\chi(z, Z): z\in Z\}$ is called the {\it character of a pre-topological space $(Z, \tau)$}, which is denoted by $\chi(Z)$.

\begin{prob}\label{pr1}
Let $H$ be a $T_{2}$ pre-topological space. Does the inequality $|H|\leq d(H)^{\chi(H)}$ hold?
\end{prob}

\begin{prob}\label{pr0}
For any two pre-continuous mappings $h$, $g$ of a pre-topological space $Y$ into a $T_{2}$ pre-topological space $Z$, is the set $\{y\in Y: h(y)=g(y)\}$ closed in $Z$?
\end{prob}

The answer to Problem~\ref{pr0} is negative. Indeed, let $(Z, \tau)$ be the pre-topological space of Example~\ref{e1}, where $Z=\{z_{1}, z_{2}, z_{3}, z_{4}\}$. Clearly $(Z, \tau)$ is $T_{2}$. Moreover, the mappings $h=id_{X}: Z\rightarrow Z$ and $g: Z\rightarrow Z$ defined by $g(z_{1})=z_{1}, g(z_{2})=z_{2}, g(z_{3})=z_{3}, g(z_{4})=z_{1}$ are pre-continuous. However, the set $\{z\in Z: h(z)=g(z)\}=\{z_{1}, z_{2}, z_{3}\}$ is not closed in $Z$ since $\{z_{4}\}$ is not open in $Z$. Moreover, the following proposition gives an another answer to Problem~\ref{pr0}.

\begin{prop}
If $h$, $g$ are two pre-continuous mappings of a topological space $Y$ into a $T_{2}$ pre-topological space $Z$, then $\{t\in Y: h(t)=g(t)\}$ is closed in $Z$.
\end{prop}

\begin{proof}
It need to prove that $G=\{t\in Y: h(t)\neq g(t)\}$ is open in $Y$. Indeed, since $h(t)\neq g(t)$ for each $t\in G$, it follows that there exist open sets $V$ and $W$ in $Z$ such that $h(t)\in V, g(t)\in W$ and $V\cap W=\emptyset$. Hence $y\in h^{-1}(V)\cap g^{-1}(W)$ is open in $Y$ and $h^{-1}(V)\cap g^{-1}(W)\subseteq G$ since $V\cap W=\emptyset$. Thus $\{t\in Y: h(t)=g(t)\}$ is closed in $Y$.
\end{proof}

\begin{prop}
For any two pre-continuous mappings $h$, $g$ of a pre-topological space $Y$ into a $T_{2}$ pre-topological space $Z$, if $h^{-1}(W)=g^{-1}(W)$ for any open set $W$ in $Z$, then $h\equiv g$; thus $\{t\in Y: h(t)=g(t)\}$ is closed in $Y$.
\end{prop}

\begin{proof}
Assume the set $A=\{t\in Y; h(t)\neq g(t)\}$ is nonempty. Then take an arbitrary $t\in A$. Hence $h(t)\neq g(t)$; then there are disjoint open sets $V$ and $O$ in $Z$ so that $h(t)\in V, g(t)\in O$. Therefore, $$t\in h^{-1}(V)\cap g^{-1}(O)=g^{-1}(V)\cap g^{-1}(O)=g^{-1}(V\cap O)=\emptyset,$$ which is a contradiction. Hence $h\equiv g$.
\end{proof}

\begin{defi}
Let $Z$ be a $T_{1}$ pre-topological space. We say that $Z$ is a {\it $T_{3}$ pre-topological space}, or a {\it regular space}, if for every $z\in Z$ and every closed set $A$ of $Z$ with $z\not\in A$ there are open subsets $V$ and $O$ such that $V\cap O=\emptyset$, $z\in V$ and $A\subseteq O$.
\end{defi}

\begin{rem}
(1) Each regular pre-topological space is $T_{2}$; however, the inverse is not true. Indeed, let $Z=\{z_{1}, z_{2}, z_{3}, z_{4}\}$ be endowed with the following pre-topology
$$\tau=\{\emptyset, \{z_{1}, z_{2}\}, \{z_{1}, z_{3}\}, \{z_{1}, z_{4}\},\{z_{2}, z_{4}\}, \{z_{2}, z_{3}\}, \{z_{1}, z_{2}, z_{4}\}, \{z_{1}, z_{2}, z_{3}\}, \{z_{1}, z_{3}, z_{4}\}, $$$$\{z_{2}, z_{3}, z_{4}\}, Z\}.$$ Then $(Z, \tau)$ is a $T_{2}$ pre-topological space. However, $(Z, \tau)$ is not regular since $\overline{\{z_{1}, z_{2}\}}=Z$.

(2) For a regular topological space $Z$, it follows that $w(Z)\leq 2^{d(Z)}$. However, there is a regular pre-topological space $Z$ with $w(Z)> 2^{d(Z)}$. Indeed, the pre-topological space in Example~\ref{e2} is regular; but $w(Z)\geq |Z|>2^{8}$ and $2^{d(Z)}=8$.
\end{rem}

\begin{prop}
Let $\mathcal{B}$ be a fixed pre-base for a $T_{1}$ pre-topological space $Z$. Then $Z$ is regular iff for every $z\in Z$ and each open neighborhood $W$ of $z$ there is an open set $O$  with $z\in O\subseteq\overline{O}\subset W$.
\end{prop}

\begin{defi}
Let $Z$ be a $T_{1}$ pre-topological space. Then $Z$ is  a {\it $T_{3\frac{1}{2}}$ pre-topological space}, or a {\it completely regular pre-topological space}, or a {\it Tychonoff pre-topological space}, provide for each $z\in Z$ and each closed subset $C\subseteq Z$ with $z\not\in C$ there exists a pre-continuous mapping $r: Z\rightarrow I$ so that $r(z)=0$ and $r(x)=1$ for each $x\in C$.
\end{defi}

Clearly, there exists a regular pre-topological space that is not Tychonoff.

\begin{prop}
A $T_{1}$ pre-topological space $Z$ is Tychonoff iff for each $z\in Z$ and each open neighborhood $V$ of $z$ in a fixed pre-base $\mathcal{B}$ there is a pre-continuous function $h_{V}: Z\rightarrow I$ such that $h_{V}(z)=0$ and $h_{V}(y)=1$ for any $y\in Z\setminus V$.
\end{prop}

\begin{defi}
A $T_{1}$ pre-topological space $Z$ is called {\it a $T_{4}$ pre-topological space}, or {\it a normal pre-topological space}, provided for each any two disjoint closed subsets $C, D\subseteq Z$ there exist disjoint two open sets $U, W$ in $Z$ so that $C\subseteq U$ and $D\subseteq W$.
\end{defi}

\begin{rem}
(1) A $T_{1}$ pre-topological space is normal iff for each closed set $C$ and each open subset $W$ which contains $C$ there exists an open subset $O$ such that $C\subseteq O\subseteq \overline{O}\subseteq W$.

(2) Each $T_{4}$ pre-topological space is $T_{3}$. By the following theorem, we see that each $T_{4}$ pre-topological space is $T_{3\frac{1}{2}}$, but the inverse is not true.
\end{rem}

By using similar methods in \cite[Theorem 1.5.11]{E1989}, we can prove the following versions of Urysohn's lemma and Tietze extension theorem (Theorems~\ref{t11} and~\ref{t21}, respectively) in the class of pre-topological spaces, thus we do not give out their proofs.

\begin{thm}\label{t11}
Let $Z$ be a normal pre-topological space. If $C, D$ are two disjoint closed subsets of $Z$, then there exists a pre-continuous function $r: Z\rightarrow I$ with $r(C)\subseteq\{0\}$ and $r(D)\subseteq\{1\}$.
\end{thm}

By the following lemma, we have the following corollary.

\begin{lem}\label{l2}
Let $\{h_{i}\}$ be a sequence of pre-continuous functions from a pre-topological space $Z$ to $\mathbb{R}$ (from $Z$ to $I$) such that $\{h_{i}\}$ is uniformly convergent to a real-valued function $h$. Then $h$ is a pre-continuous function from $Z$ to $\mathbb{R}$ (from $Z$ to $I$).
\end{lem}

A subset $G$ of a pre-topological space $Z$ is called {\it a $G_{\delta}$-set} ({\it resp., a $F_{\sigma}$-set}) if $G$ is the intersections of countably many open subsets (resp., the unions of countably many closed subsets).

\begin{cor}\label{c1}
Let $C$ be a subset of a normal pre-topological space $Z$. Then $C$ is a closed $G_{\delta}$-set (resp. an open $F_{\sigma}$-set) iff there is a pre-continuous function $r: Z\rightarrow I$ with $C=r^{-1}(0)$ (resp. $C=r^{-1}((0, 1])$).
\end{cor}

Taking complements, we can restate Corollary~\ref{c1} as follows:

\begin{cor}
Let $O$ be a subset of a normal pre-topological space $Z$. Then $O$ is an open $F_{\sigma}$-set iff there is a pre-continuous function $r: Z\rightarrow I$ with $O=r^{-1}((0, 1])$.
\end{cor}

\begin{ex}
There exists a regular, non-normal and finite pre-topological space $Z$.
\end{ex}

\begin{proof}
Let $(Z, \tau)$ be a pre-topology such that $\omega>|Z|>4$ and the family
$$\mathscr{P}=\{\{y, z\}: y\in\{a_{1}, a_{2}, a_{3}\}, z\in Z, y\neq z, \{y, z\}\neq\{a_{1}, a_{2}\}\}$$ is a pre-base of $(Z, \tau)$, where $a_{1}, a_{2}, a_{3}$ are three distinct points of $Z$.
Obviously, $(Z, \tau)$ is a  regular and finite pre-topological space. However, $Z$ is not normal. Indeed, let $D=\{a_{1}, a_{2}\}$ and $F=Z\setminus\{a_{1}, a_{2}, a_{3}\}$. Then $D$ and $F$ are closed in $Z$ and $D\cap F=\emptyset$. Clearly, neither $D$ nor $F$ are open in $Z$. Assume that there exist open sets $O$ and $W$ such that $D\subseteq O$, $F\subseteq W$ and $O\cap W=\emptyset$. But since neither $D$ nor $F$ are open in $Z$, we conclude that $a_{3}\in O$ and $a_{3}\in W$, a contradiction.
\end{proof}

\begin{prob}
When a regular pre-topological space is normal?
\end{prob}

\begin{thm}\label{t21}
Each pre-continuous function from a closed set $A$ of a normal pre-topological space $G$ to $\mathbb{R}$ or $I$ can pre-continuously extendable over $G$.
\end{thm}

\maketitle
\subsection{Connected pre-topological spaces}
\
\newline
\indent In the study of topological spaces, the connectedness is an important concept. By a similar definition of connectedness in topological space, we give the concept of connectedness in the class of pre-topological spaces.

\begin{defi}
Let $Z$ be a pre-topological space. A pair $O, W$ of two disjoint nonempty open sets of $Z$ is a {\it separation} of $Z$ whenever $Z=O\cup W$. The pre-topological space $Z$ is called {\it connected} provided there is no separation of $Z$.
\end{defi}

\begin{rem}
(1) If a pre-topological space $Z$ is connected, then so is any pre-topological space that is pre-homeomorphic to $Z$.

(2) A pre-topological space $Z$ is connected iff any clopen subset is either $\emptyset$ or $Z$.
\end{rem}

We don't give the proofs of some results in this subsection, the reader can see the proofs in \cite{lijinjin2007}.

\begin{lem}\label{l1}
Let $Z$ be a subset of pre-topological space $Y$. Then $Z$ is connected iff there exists two disjoint non-empty subsets $D$ and $T$ such that $Z=D\cup T$, $\overline{D}\cap T=\emptyset$ and $\overline{T}\cap D=\emptyset$.
\end{lem}

The following lemma is easily checked.

\begin{lem}\label{l0}
Let $C$ and $D$ be a separation of pre-topological space $Z$ and $G$ be a connected subspace of $Z$. Then $G\subseteq C$ or $G\subseteq D$.
\end{lem}

Then we have the following four theorems.

\begin{thm}\label{t7}
Let $Z$ be a pre-topological spaces and $\{C_{i}\}_{i\in \gamma}$ be a family of non-empty connected subspaces. If $\bigcap_{i\in \gamma}C_{i}\neq\emptyset$, then $\bigcup_{i\in \gamma}C_{i}$ is connected.
\end{thm}

\begin{thm}\label{t9}
Let $C$ be a connected subspace of a pre-topological space $Z$. If $C\subseteq A\subseteq \overline{C}$, then $A$ is also connected.
\end{thm}

\begin{thm}\label{t6}
Let $h: Y\rightarrow Z$ be a pre-continuous mapping, where $Y$ and $Z$ are pre-topological spaces. If $Y$ is connected, then $Z$ is connected.
\end{thm}

\begin{thm}\label{t8}
Let $X_{1}$ and $X_{2}$ be two connected pre-topological space. Then the product $X_{2}\times X_{2}$ is connected.
\end{thm}

\begin{thm}
The product pre-topology $\prod_{\alpha\in J}X_{\alpha}$, where each $X_{\alpha}\neq\emptyset$, is connected iff each $X_{\alpha}$ is connected.
\end{thm}

The family $\{W_{s}\}_{s\in S}$ of subsets of a set $Z$ is said to be a cover of $Z$ if $\bigcup_{s\in S}W_{s}=Z$. Further, if $Z$ is a pre-topological space and each the set $W_{s}$ is open, then the family $\{W_{s}\}_{s\in S}$ is said to be {\it open}. The following theorem is new.

\begin{thm}\label{t14}
A pre-topological space $Z$ is connected iff for each open cover $\{U_{\alpha}\}_{\alpha\in J}$ of $Z$ and every pair $x_{1}, x_{2}$ of points of $Z$ there exists a finite sequence $\alpha_{1}, \alpha_{2}, \cdots, \alpha_{k}$ of elements of $J$ so that $x_{1}\in U_{\alpha_{1}}, x_{2}\in U_{\alpha_{k}}$ and $U_{\alpha_{i}}\cap U_{\alpha_{j}}\neq\emptyset$ iff $|i-j|\leq 1$.
\end{thm}

\begin{proof}
Sufficiency. Let $Z$ be not a connected space; hence we can fix a separation $U$ and $W$, thus $\{U, W\}$ is an open cover of $Z$. Take arbitrary $x\in U$ and $y\in W$. However, $U$ intersects with $W$ is empty, a contradiction. Therefore, $Z$ is connected.

Necessity. Let $Z$ be connected, and let $\mathcal{O}=\{U_{\alpha}\}_{\alpha\in J}$ be an open cover of $Z$. Take any pair $x_{1}, x_{2}$ of $Z$. Let $\mathcal{A}=\{W: x_{1}\in W, W\in\{U_{\alpha}\}_{\alpha\in J}\}$ and $\mathcal{B}=\{W: x_{2}\in W, W\in\{U_{\alpha}\}_{\alpha\in J}\}$. If $(\bigcup\mathcal{A})\cap (\bigcup\mathcal{B})\neq\emptyset$, then we can choose $U_{\alpha_{1}}\in\mathcal{A}$ and $U_{\alpha_{2}}\in\mathcal{B}$ such that $x_{1}\in U_{\alpha_{1}}$, $x_{2}\in U_{\alpha_{2}}$ and $U_{\alpha_{1}}\cap U_{\alpha_{2}}\neq\emptyset$. Then the proof complete. Hence we may assume that $(\bigcup\mathcal{A})\cap (\bigcup\mathcal{B})=\emptyset$. Let $\mathcal{A}_{1}=\{U\in\mathcal{O}: U\cap (\bigcup\mathcal{A})\neq\emptyset\}\setminus \mathcal{A}$. If $(\bigcup\mathcal{A}_{1})\cap(\bigcup\mathcal{B})\neq\emptyset$. Then there exist $U_{1}\in\mathcal{A}, U_{2}\in\mathcal{A}_{1}$ and $U_{3}\in \mathcal{B}$ such that $U_{1}\cap U_{2}\neq\emptyset$ and $U_{2}\cap U_{3}\neq\emptyset$. Then the proof complete. Otherwise, assume that we can take a finite family $\mathcal{A}_{i}\subseteq \mathcal{O}$ $(i\leq n)$ such that the following conditions hold.

\smallskip
(1) For every $i\leq n$, $$\mathcal{A}_{i}=\{W\in\mathcal{O}: W\cap (\bigcup\mathcal{A}_{i-1})\neq\emptyset\}\setminus(\mathcal{A}\cup\bigcup_{k<i}\mathcal{A}_{k}).$$

\smallskip
(2) For every $2\leq i\leq n$, $(\bigcup\mathcal{A}_{i-1})\cap (\bigcup\mathcal{A}_{i})\neq\emptyset.$

\smallskip
(3) For every $2\leq i\leq n-1$, $(\bigcup_{j\leq i}\bigcup\mathcal{A}_{j})\cap (\bigcup\mathcal{B})=\emptyset.$

If $(\bigcup\mathcal{A}_{n})\cap (\bigcup\mathcal{B})\neq\emptyset$, then by our construction of the family $\{\mathcal{A}_{i}: i\leq n\}$ it is easily verified that theorem holds. Then put $\mathcal{A}_{i}=\emptyset$ for any $i>n$. Otherwise, put $$\mathcal{A}_{n+1}=\{U\in\mathcal{O}: U\cap \bigcup\mathcal{A}_{n}\neq\emptyset\}\setminus(\mathcal{A}\cup\bigcup_{i\leq n}\mathcal{A}_{i}).$$
By induction, we have a sequence $\{\mathcal{A}_{i}: i\in\omega\}$, where $\mathcal{A}_{0}=\mathcal{A}$, of subsets of $\mathcal{O}$ such that the following conditions hold.

\smallskip
(a) For each $n\in\mathbb{N}$, $$\mathcal{A}_{n}=\{U\in\mathcal{O}: U\cap \bigcup\mathcal{A}_{n-1}\neq\emptyset\}\setminus(\mathcal{A}\cup\bigcup_{i<n}\mathcal{A}_{i});$$

\smallskip
(b) for any $2\leq n$, $\bigcup\mathcal{A}_{n-1}\cap \bigcup\mathcal{A}_{n}\neq\emptyset;$

\smallskip
(c) for any $2\leq n$, if $(\bigcup_{i\leq n}\bigcup\mathcal{A}_{i})\cap \bigcup\mathcal{B}\neq\emptyset,$ then $\mathcal{A}_{m}=\emptyset$ for $m>n$.

We claim that there is $n\in\mathbb{N}$ such that $\mathcal{A}_{n}=\emptyset$. Suppose not, $\mathcal{A}_{n}\neq\emptyset$ for any $n\in\omega$, thus $\bigcup\mathcal{A}_{n}\cap \bigcup\mathcal{B}=\emptyset$. However, we conclude that $\bigcup_{n\in\omega}\bigcup\mathcal{A}_{n}=Z$, which leads to a contradiction. Indeed, put $\mathcal{D}=\{O\in\mathcal{O}: O\cap \bigcup_{n\in\omega}\bigcup\mathcal{A}_{n}=\emptyset\}$. Then $\bigcup\mathcal{D}$ is open in $Z$ and $\bigcup\mathcal{D}\cup\bigcup_{n\in\omega}\bigcup\mathcal{A}_{n}=Z$. From the connectedness of $Z$, it follows that $\bigcup\mathcal{D}=\emptyset$.
\end{proof}

\begin{defi}
Let $C$ be a subset of a pre-topological space $Z$. The set $C$ is {\it chain connected} in $Z$, if for every open covering $\mathcal{U}$ in $Z$ and any $x, y\in C$, there can find a finite sequence $U_{1}, U_{2}, \ldots, U_{n}$ of $\mathcal{U}$, such that $U_{i}\cap U_{i+1}\neq\emptyset$ for any $i=1, 2, \ldots, n-1$, $x\in U_{1}$ and $y\in U_{n}$. If $C=Z$, we say that $Z$ is chain-connected.
\end{defi}

From Theorem~\ref{t14}, we have the following corollary.

\begin{cor}
Each connected pre-topological space is chain-connected.
\end{cor}

Let $Z$ be the real line $\mathbb{R}$, and let $\mathscr{B}=\{(a, +\infty): a\in\mathbb{R}\}\cup\{(-\infty, a): a\in\mathbb{R}\}.$ Assume $\mathscr{F}$ is the usual topology on $Z$. Put $$\mathscr{F}^{\prime}=\{U\in \mathscr{F}: U=\bigcup\mathscr{B}^{\prime}\ \mbox{for soem}\ \mathscr{B}^{\prime}\ \mbox{of}\ \mathscr{B}\}.$$Clearly, $(\mathbb{R}, \mathscr{F}^{\prime})$ is a connected pre-topological space that is coarser than $(\mathbb{R}, \mathscr{F})$.

\begin{defi}
Let $Z$ be a pre-topological space. If for any $y, z\in Z$ there exists a pre-continuous function $r: I\rightarrow Z$ such that $r(0)=y$ and $r(1)=z$, then $Z$ is said to be a {\it pathwise connected space}, where $I$ with the usual topology.
\end{defi}

Clearly, we have the following two propositions.

\begin{prop}
Each pathwise connected pre-topological space is connected.
\end{prop}

Note that a connected pre-topological space may not be pathwise connected.

\begin{prop}
The image of pathwise connected pre-topological space under pre-continuous mapping is pathwise connected as well.
\end{prop}

By Theorem~\ref{t10}, the following theorem is easily checked.

\begin{thm}
The product pre-topology $\prod_{\alpha\in J}X_{\alpha}$, where each $X_{\alpha}\neq\emptyset$, is pathwise connected iff all spaces $X_{\alpha}$ are pathwise connected.
\end{thm}

\begin{defi}
Let $(Z, \tau)$ be a finite pre-topological space. We say that $Z$ is {\it $n$-connected} provided for any distinct open sets $U$ and $W$ there are open subsets $O_{0}, O_{1}, \ldots, O_{m}$ such that $O_{0}=U, O_{1}, \ldots, O_{m}=W$ and $|O_{i}\triangle O_{i+1}|=n$ for any $0\leq i\leq m-1$.

We say that $Z$ is {\it tight $n$-connected} provided for any distinct open sets $V$ and $W$ there exist open subsets $O_{0}, O_{1}, \ldots, O_{m}$ such that $O_{0}=V, O_{1}, \ldots, O_{m}=W$ and $|O_{i}\triangle O_{i+1}|=n$ for any $0\leq i\leq m-1$, where $m=|V\triangle W|$.
\end{defi}

From the definition, the following proposition holds.

\begin{prop}
If $Z$ is an $1$-connected pre-topological space, then $W^{\mathcal{LC}}\neq\emptyset$ for every open set $W$ in $Z$.
\end{prop}

\begin{thm}\label{t12}
If $Z$ is a finite pre-topological space, then the following are equivalent.
\begin{enumerate}
\smallskip
\item $Z$ is tight 1-connected;

\smallskip
\item $(U\triangle W)\cap U^{\mathcal{LC}}\neq\emptyset$ for any two distinct open sets $U$ and $W$;

\item any two open sets $U$ and $W$ which satisfy $U^{\mathcal{I}}\subseteq W$, $U^{\mathcal{O}}\subseteq Z\setminus W$ must be equal.
\end{enumerate}
\end{thm}

\begin{proof}
(1) $\Rightarrow$ (2). Take any two distinct open sets $U$ and $W$. Since $Z$ is tight 1-connected, there exist open subsets $O_{0}, O_{1}, \ldots, O_{m}$ such that $O_{0}=U, O_{1}, \ldots, O_{m}=W$ and $|O_{i}\triangle O_{i+1}|=n$ for any $0\leq i\leq m-1$, where $|U\triangle W|=m$. Hence it is easily verified that $U\cap W\subseteq O_{1}\subseteq U\cup W$. Moreover, $U$ and $O_{1}$ differ by exactly one element $x$, then $x\in U$ or $x\in W$, but not both. Hence, $x\in (U\triangle W)\cap U^{\mathcal{LC}}$.

(2) $\Rightarrow$ (3). Suppose not, then there exist distinct open sets $U$ and $W$ such that $U^{\mathcal{I}}\subseteq W$ and $U^{\mathcal{O}}\subseteq Z\setminus W$. Take any $z\in (U\triangle W)\cap U^{\mathcal{LC}}$. If $z\in U$, then $z\in U^{\mathcal{I}}\subseteq W$, which contradicts $q\in U\triangle W$. Therefore, $z\not\in U$, but then $z\in W\cap U^{\mathcal{O}}$, hence $z\in U^{\mathcal{O}}\subseteq Z\setminus W$ and $z\in W$, a contradiction.

(3) $\Rightarrow$ (1). Let $U$ and $W$ be distinct open subsets in $Z$ with $|U\bigtriangleup W|=m>0$. Since $U\neq W$, from our assumption it follows that $$U^{\mathcal{I}}\nsubseteq W\ \mbox{or}\ U^{\mathcal{O}}\nsubseteq Z\setminus W.$$Hence there exists an element $z\in Z$ with $$z\in (U^{\mathcal{I}}\setminus W)\cup (U^{\mathcal{O}}\setminus (Z\setminus W)).$$If $z\in U^{\mathcal{I}}\setminus W$, we put $W_{1}=U\setminus\{z\}$; if $z\in U^{\mathcal{O}}\setminus (Z\setminus W)$, we put $W_{1}=U\cup\{z\}$. Then we have $|W_{1}\triangle W|=m-1$. The result follows by induction.
\end{proof}

\begin{ex}\label{e4}
There exists a tight 1-connected and non-connected pre-topological space $Z$.
\end{ex}

\begin{proof}
Let $Z=\{z_{1}, z_{2}, z_{3}, z_{4}\}$ be endowed with the following pre-topology $$\mathscr{H}=\{\emptyset, \{z_{1}\}, \{z_{4}\}, \{z_{3}, z_{4}\}, \{z_{1}, z_{2}\}, \{z_{1}, z_{4}\}, \{z_{1}, z_{2}, z_{3}\}, \{z_{1}, z_{3}, z_{4}\}, $$$$\{z_{1}, z_{2}, z_{4}\}, \{z_{2}, z_{3}, z_{4}\}, Z\}.$$ It is easily checked that $Z$ is tight 1-connected. However, $Z$ is not connected since $\{z_{1}, z_{2}\}$ is open and closed in $Z$.
\end{proof}

\begin{ex}\label{e3}
There exists a connected pre-topological space $Z$ that is not tight 1-connected.
\end{ex}

\begin{proof}
Let $Z=\{z_{1}, z_{2}, z_{3}, z_{4}, z_{5}\}$ endowed with the following pre-topology $$\mathscr{H}=\{\emptyset, \{z_{1}\}, \{z_{1}, z_{2}, z_{4}\}, \{z_{1}, z_{2}, z_{3}\}, \{z_{1}, z_{3}, z_{4}\}, \{z_{1}, z_{2}, z_{3}, z_{4}\}, Z\}.$$ Clearly, $Z$ is connected. However, $Z$ is not tight 1-connected since there is not intermediate open set between $\{z_{1}\}$ and $\{z_{1}, z_{2}, z_{3}\}$.
\end{proof}

\begin{rem}
It is well known that the cardinality of a $T_{1}$ connected topological space is infinite. However, there exists a $T_{1}$ connected pre-topological space which is finite. Indeed, let $Z=\{z_{1}, z_{2}, z_{3}, z_{4}\}$ endowed with the following pre-topology $$\mathscr{H}=\{\emptyset, \{z_{1}, z_{2}, z_{4}\}, \{z_{1}, z_{2}, z_{3}\}, \{z_{1}, z_{3}, z_{4}\}, \{z_{2}, z_{3}, z_{4}\}, Z\}.$$ Then $Z$ is a $T_{1}$ connected pre-topological space.
\end{rem}

\maketitle
\section{The language of pre-topology in knowledge spaces}
In this section, we will discuss the language of pre-topology in the theory of knowledge spaces. From Section 2, we see that knowledge space is just  pre-topological space discussed above in the present paper. First, we list some terminology in the theory of pre-topological space and knowledge space respectively that are equivalent.

\smallskip
$$\begin{tabular}{lccc}

\hline
 &{\bf pre-topological space}&{\bf knowledge space}&{\bf relation}\\
 \cline{1-4}
1&open set&knowledge state&$\Leftrightarrow$\\
 \cline{1-4}
2&$T_{0}$-space&discriminative&$\Leftrightarrow$\\
 \cline{1-4}
3&$T_{1}$-space&bi-discriminative&$\Leftrightarrow$\\
 \cline{1-4}
4&atom pre-base&base&$\Rightarrow$\\
 \cline{1-4}
5&subspace&projection&$\Leftrightarrow$\\
 \cline{1-4}
6&$T_{0}$-reduction&discriminative reduction&$\Leftrightarrow$\\
 \cline{1-4}
7&Alexandroff space&quasi ordinal space&$\Leftrightarrow$\\
 \cline{1-4}
8&$T_{0}$-Alexandroff space&ordinal space&$\Leftrightarrow$\\
 \cline{1-4}
9&locally inner closed points&inner fringe&$\Leftrightarrow$\\
 \cline{1-4}
10&locally outer closed points&outer fringe&$\Leftrightarrow$\\
 \cline{1-4}
11&locally closed points&fringe&$\Leftrightarrow$\\
 \cline{1-4}
12&tight 1-connected&well-graded&$\Leftrightarrow$\\
 \cline{1-4}
13&pre-quotient pre-topology&discriminative reduction&$\Leftrightarrow$\\
\hline
\end{tabular}$$
\smallskip

We always say that ($Q, \tau_{\mathscr{H}}$) is a pre-topological space with $\tau_{\mathscr{H}}=\mathscr{H}$ for a knowledge space ($Q, \mathscr{H}$).

\maketitle
\subsection{The applications of axioms of separation in knowledge spaces}
\
\newline
\indent In this subsection, we discuss some applications of axioms of separation in knowledge spaces. First, the following two theorems in \cite{XGLJ} are provided.

\begin{thm}\cite{XGLJ}\label{t19}
Let ($Q, \mathscr{H}$) be a knowledge space. Then ($Q, \tau_{\mathscr{H}}$) is $T_{0}$ iff ($Q, \mathscr{H}$) is discriminative.
\end{thm}

\begin{thm}\cite{XGLJ}
Let ($Q, \mathscr{H}$) be a knowledge space. Then ($Q, \tau_{\mathscr{H}}$) is $T_{1}$ iff ($Q, \mathscr{H}$) is bi-discriminative.
\end{thm}

\begin{defi}\cite{falmagne2011learning}
Assume that $(Q, \mathscr{H})$ is a discriminative knowledge structure. For each $H\in\mathscr{H}$, we say that the set $H^{\mathcal{I}}=\{t\in H: H\setminus\{t\}\in\mathscr{H}\}$ is the {\it inner fringe} of $H$, and that the set $H^{\mathcal{O}}=\{t\in Q\setminus H: H\cup\{t\}\in\mathscr{H}\}$ is the {\it outer fringe} of $H$. Moreover, the set $$H^{\mathcal{F}}=H^{\mathcal{I}}\cup H^{\mathcal{O}}$$ is said to be the {\it fringe} of $H$.
\end{defi}

The following theorem shows that we can give a characterization for the bi-discriminative knowledge spaces by inner fringe of $Q$. Since the proof is easy, we do not give out the proof.

\begin{thm}
For a knowledge space ($Q, \mathscr{H}$), it is bi-discriminative iff $Q\setminus\{t\}\in\mathscr{H}$ for each $t\in Q$, that is, $Q^{\mathcal{I}}=Q$.
\end{thm}

Moreover, the following proposition shows that we can give a characterization for knowledge spaces by inner fringe of knowledge states.

\begin{prop}\label{c3}
Let ($Q, \mathscr{H}$) be a knowledge space. For each $H\in\mathscr{H}$ and $t\in H$, it follows that $t\in H^{\mathcal{I}}$ iff $q\not\in\overline{Q\setminus (H\setminus\{t\})}$ for any $q\in H\setminus\{t\}$.
\end{prop}

\begin{proof}
Necessity. Assume that $t\in H^{\mathcal{I}}$, then $H\setminus\{t\}\in\mathscr{H}$, hence $(H\setminus\{t\})\cap (Q\setminus (H\setminus\{t\}))=\emptyset$, that is, $q\not\in\overline{Q\setminus (H\setminus\{t\})}$ for any $q\in H\setminus\{t\}$.

Sufficiency. Assume that $q\not\in\overline{Q\setminus (H\setminus\{t\})}$ for any $q\in H\setminus\{t\}$, then for each $q\in H\setminus\{t\}$ there exists $G(q)\in\mathscr{H}$ such that $G(q)\cap (Q\setminus (H\setminus\{t\}))=\emptyset$, hence $G(q)\subseteq H\setminus\{t\}$. Therefore, $H\setminus\{t\}=\bigcup_{q\in H\setminus\{t\}}G(q)\in\mathscr{H}.$
\end{proof}

The following proposition shows that we can give a characterization of outer fringe of a state for knowledge spaces by the derived points.

\begin{prop}\label{p8}
Let ($Q, \mathscr{H}$) be a knowledge space, $H\in\mathscr{H}$ and $t\in Q\setminus H$. Then $t\in H^{\mathcal{O}}$ iff $t\not\in (Q\setminus H)^{d}$ in $(Q, \tau_{\mathscr{H}})$.
\end{prop}

\begin{proof}
Necessity. Assume $t\in H^{\mathcal{O}}$, then $H\cup\{t\}\in \mathscr{H}$, hence $$(H\cup\{t\})\cap \left((Q\setminus H)\setminus\{t\}\right)=\emptyset,$$ thus $t\not\in (Q\setminus H)^{d}$ in $(Q, \tau_{\mathscr{H}})$.

Sufficiency. Suppose that $t\not\in (Q\setminus H)^{d}$ in $(Q, \tau_{\mathscr{H}})$, then there exists $G\in\mathscr{H}$ such that $t\in G$ and $G\cap \left((Q\setminus H)\setminus\{t\}\right)=\emptyset$. Hence $L\subseteq H\cup\{t\}$, hence $H\cup G=H\cup\{t\}\in \mathscr{H}$, that is, $t\in H^{\mathcal{O}}$.
\end{proof}

By Propositions~\ref{c3} and~\ref{p8}, the following theorem holds.

\begin{thm}
Let ($Q, \mathscr{H}$) be a knowledge space, $H\in\mathscr{H}$ and $t\in Q$. Then $t\in H^{\mathcal{F}}$ iff either $(H\setminus\{t\})\cap\overline{Q\setminus (H\setminus\{t\})}=\emptyset$ and $t\in H$ or $t\not\in (Q\setminus H)^{d}$ in $(Q, \tau_{\mathscr{H}})$ and $t\not\in H$.
\end{thm}

In knowledge assessment, the device of the items is very important. Since the time, the manpower and the material resources are limited, ones hope that the process of the assessment is efficient, and that the items are not repeated as far as possible in the process of the assessment. Further, ones hope that the selection of items in the knowledge assessment should not only have depth but also breadth. However, if the device of the items is not appropriate, then it is possible that the process of the knowledge assessment is invalid, see the following example.

\begin{ex}\label{ex3}
Let ($Q, \mathscr{H}$) be a knowledge space, where $Q=\{a_{1}, a_{2}, a_{3}, a_{4}, a_{5}\}$ and
$$\mathscr{H}=\{\emptyset\}\cup\{O: |Q\setminus O|\leq 2, O\subseteq Q\}.$$ Then ($Q, \mathscr{H}$) is a bi-discriminative. In a knowledge assessment, a teacher in order to check the level for a course of a student, this teacher designs these five items (that is, knowledge points). Assume this student is capable of solving of items $\{a_{1}, a_{2}\}$ (that is, she$/$his knowledge state indeed). For each $K\in\mathscr{H}\setminus\{\emptyset\}$, the teacher devices an exercise $E_{K}$. Clearly, the student can not complete any exercise; then it is problem for the teacher to check the level for a course of this student.
\end{ex}

In a knowledge assessment devices, we always find that the methods are inadequate, hence ones have to provide further deepening and enriching the methods in the knowledge assessment. Therefore, we have the following subclass of bi-discriminative knowledge spaces.

\begin{defi}
A knowledge space ($Q, \mathscr{H}$) is {\it completely discriminative} provided there exists $H\in \mathscr{H}_{p}$ and $L\in \mathscr{H}_{q}$ with $H\cap L=\emptyset$ for any distinct $p, q\in Q$.
\end{defi}

The \cite[Example 3]{XGLJ} is a completely discriminative knowledge space. Obviously, Example~\ref{ex3} is a bi-discriminative knowledge space which is not completely discriminative. If we enrich the methods of knowledge assessment in Example~\ref{ex3} as follows$$\mathscr{H}=\{\emptyset\}\cup\{U: |Q\setminus U|\leq 3, U\subset Q\},$$then we will know the level for the course of this student.

\begin{defi}\cite{doignon2011knowledge}
Let $\mu: Q\rightarrow 2^{2^{S}\setminus\{\emptyset\}}\setminus\{\emptyset\}$ be a mapping, where $Q$ and $S$ are non-empty sets of items and skills respectively. Then a triple $(Q, S, \mu)$ is called a {\it skill multimap}. The elements of $\mu(t)$ are said to be {\it competencies} for every $t\in Q$. Moreover, if the elements of each $\mu(t)$ are pairwise incomparable, then $(Q, S, \mu)$ is called a {\it skill function}. 
\end{defi}

\begin{defi}\cite{doignon2011knowledge}
For a skill multimap $(Q, S, \mu)$, we define a mapping $p: 2^{S}\rightarrow 2^{Q}$ as follows: for each $R\in 2^{S}$, put $$p(R)=\{g\in Q|\mbox{there exists}\ R_{g}\in\mu(g)\ \mbox{with}\ R_{g}\subseteq R\}.$$ Then, the mapping $p$ is said to be {\it the problem function induced by $(Q, S, \mu)$}.
Put $$\mathscr{H}=\{p(R)| R\in 2^{S}\}.$$We say that the knowledge structure $(Q, \mathscr{H})$ is {\it delineated by $(Q, S, \mu)$}.
\end{defi}

In \cite[p117, Problem 6]{falmagne2011learning}, the authors posed the following problem.

\begin{prob}
Under which condition on a skill multimap is the delineated structure a knowledge space?
\end{prob}

Indeed, in \cite{XGLJ} the authors also asked which kind of skill multimaps delineate knowledge spaces. The following theorem gives an answer to this problem when each item with finitely many competencies. For a skill multimap $(Q, S, \mu)$ and each $t\in Q$, we denote $\mu_{M}(t)$ by the set of minimum elements in $\mu(t)$.

\begin{thm}\label{ttttt2}
Let $(Q, S, \mu)$ be a skill multimap, where each $\mu(t)$ is a finite set. Then the delineate knowledge structure $(Q, \mathscr{H})$ is a knowledge space iff, for any $H\subseteq Q$, $H\in\mathscr{H}$ iff there is $\mathcal{P}_{H}\subseteq\bigcup_{t\in Q}\mu_{M}(t)$ such that $H=\bigcup_{D\in\mathcal{P}_{H}}p(D)$.
\end{thm}

\begin{proof}
Sufficiency. Let $p$ be the problem function induced by $\mu$, and let $(Q, \mathscr{H})$ be the knowledge structure which is delineated by $(Q, S, \mu)$. Take any a subfamily $\mathscr{H}^{\prime}$ of $\mathscr{H}$. We claim that $\bigcup\mathscr{H}^{\prime}\in\mathscr{H}$. Indeed, for each $H\in\mathscr{H}^{\prime}$, there exist a subset $\mathcal{P}_{H}\subseteq \bigcup_{t\in Q}\mu_{M}(t)$ such that $\bigcup_{D\in \mathcal{P}_{H}}p(D)=H$. Put $\mathcal{P}=\bigcup_{H\in\mathscr{H}^{\prime}}\mathcal{P}_{H}$. Clearly, we have $$\bigcup\mathscr{H}^{\prime}=\bigcup_{H\in\mathscr{H}^{\prime}}\bigcup_{D\in \mathcal{P}_{H}}p(D)=\bigcup_{D\in \mathcal{P}}p(D).$$ From our assumption, it follows that $\bigcup\mathscr{H}^{\prime}\in\mathscr{H}$.

Necessity. Let $(Q, \mathscr{H})$ be the knowledge space which is delineated by $(Q, S, \mu)$. Take any $H\subseteq Q$. Assume that $H\in\mathscr{H}$. Hence we can take a subset $R\subseteq S$ such that $p(R)=H$. For each $t\in H$, there exists a $C_{t}\in \mu_{M}(t)$ such that $C_{t}\subseteq R$. We claim that $H=\bigcup_{t\in H}p(C_{t})$. Clearly, $H\subseteq\bigcup_{t\in H}p(C_{t})$; moreover, since each $C_{t}\subseteq R$ and $p(R)=H$, it follows that $\bigcup_{t\in H}p(C_{t})\subseteq H$. Therefore, $H=\bigcup_{t\in H}p(C_{t})$. Now assume that there exists $\mathcal{P}_{H}\subseteq\bigcup_{t\in Q}\mu_{M}(t)$ such that $H=\bigcup_{D\in\mathcal{P}_{H}}p(D)$. Since $(Q, \mathscr{H})$ is a knowledge space and each $p(D)\in\mathscr{H}$, it follows that $\bigcup_{D\in\mathcal{P}_{H}}p(D)\in\mathscr{H}$, hence $H\in\mathscr{H}$.
\end{proof}

\begin{cor}\label{2021}
Let $(Q, S, \mu)$ be a skill multimap such that each $\mu(t)$ is a finite set. If, for any subset $Q^{\prime}\subseteq Q$ and $g\in Q$, the following condition ($\star$) holds, then $(Q, \mathscr{H})$ is a knowledge space.

\smallskip
($\star$) For any $\mathcal{M}\subseteq\bigcup_{t\in Q^{\prime}}\mu(t)$, if $C\setminus D\neq\emptyset$ for any $C\in \mu_{M}(g)$ and $D\in \mathcal{M}$, then  $C\setminus \bigcup\mathcal{M}\neq\emptyset$ for each $C\in \mu_{M}(g)$.
\end{cor}

\begin{proof}
By Theorem~\ref{ttttt2}, we need to prove that for any $H\subseteq Q$, $H\in\mathscr{H}$ iff there exists $\mathcal{M}\subseteq\bigcup_{t\in Q}\mu_{M}(t)$ such that $H=\bigcup_{D\in\mathcal{M}}p(D)$.

Suppose that $H\in\mathscr{H}$, then there exists a subfamily $\mathcal{M}$ of $\bigcup_{t\in H}\mu_{M}(t)$ such that $H=p(\bigcup\mathcal{M})$. We claim that $H=p(\bigcup\mathcal{M})=\bigcup_{D\in\mathcal{M}}p(D)$. Clearly, it follows that $\bigcup_{D\in\mathcal{M}}p(D)\subseteq p(\bigcup\mathcal{M})$. Assume that $p(\bigcup\mathcal{M})\setminus\bigcup_{D\in\mathcal{M}}p(D)\neq\emptyset$. Take any $g\in p(\bigcup\mathcal{M})\setminus\bigcup_{D\in\mathcal{M}}p(D)$. Hence $C\setminus D\neq\emptyset$ for each $C\in \mu_{M}(g)$ and $D\in \mathcal{M}$, then from the condition ($\star$) it follows that $C\setminus \bigcup\mathcal{M}\neq\emptyset$ for any $C\in \mu_{M}(g)$, which leads to a contradiction with $g\in p(\bigcup\mathcal{M})\setminus\bigcup_{D\in\mathcal{M}}p(D)$.

Pick any $H\subseteq Q$, and assume that there is $\mathcal{M}\subseteq\bigcup_{t\in Q}\mu_{M}(t)$ such that $H=\bigcup_{D\in\mathcal{M}}p(D)$. We claim that $H\in\mathscr{H}$. Indeed, $H=p(\bigcup\mathcal{M})$ by the condition ($\star$), hence $H\in\mathscr{H}$.
\end{proof}

\begin{cor}~\label{cor1}
Let $(Q, S, \mu)$ be a skill multimap. If, for each $t\in Q$ and $C\in\mu(t)$, there exists $s_{C}\in C$ such that $\{s_{C}\}\in\mu(t)$, then the delineate knowledge structure is a knowledge space.
\end{cor}

\begin{proof}
Indeed, it is easily verified that the condition ($\star$) holds in Corollary~\ref{2021}; moreover, each $\mu_{M}(q)$ consists of elements of singleton. Therefore, the delineate knowledge structure is a knowledge space by Corollary~\ref{2021}.
\end{proof}

It follows from the following example that the condition in Corollary~\ref{cor1} is sufficient and non-necessary condition.

\begin{ex}
Assume that $(Q, S, \mu)$ is a skill multimap such that $\mu(t)=\{S\}$ for each $t\in Q$. It is obvious that the delineate knowledge structure $\mathscr{H}=\{\emptyset, Q\}$ is a knowledge space. However, the condition in Corollary~\ref{cor1} does not hold.
\end{ex}

By Corollary~\ref{2021}, the following corollary holds.

\begin{cor}
Let $(Q, S, \mu)$ be a skill function. If, for any subset $Q^{\prime}\subseteq Q$ and $g\in Q$, the following condition ($\ast$) holds, then $(Q, \mathscr{H})$ is a knowledge space.

\smallskip
($\star$) For each $\mathcal{M}\subseteq\bigcup_{t\in Q^{\prime}}\mu(t)$, if $C\setminus D\neq\emptyset$ for any $C\in \mu(g)$ and $D\in \mathcal{M}$, then  $C\setminus \bigcup\mathcal{M}\neq\emptyset$ for each $C\in \mu(g)$.
\end{cor}

\begin{defi}\cite{XGLJ}
Suppose that~$Z$ is a non-empty set,~$\mathcal{O}$ and~$\mathcal{W}$ be two family of subsets on~$Z$. If for any~$O\in \mathcal{O}$ there exists $W\in \mathcal{W}$ with $W\subseteq O$, then we say that~$\mathcal{O}$ is {\it refined by~$\mathcal{W}$} which is denoted by~$\mathcal{O}\Subset \mathcal{W}$. Otherwise,~$\mathcal{O}$ is not refined by $\mathcal{W}$ which is denoted by~$\mathcal{O}\not\Subset \mathcal{W}$. Further, if~$\mathcal{O}=\{O\}$, then we say that $O$ is {\it refined by~$\mathcal{W}$}, which is denoted by~ $O\Subset \mathcal{W}$; otherwise, we say that $O$ is not {\it refined by~$\mathcal{W}$}, which is denoted by $O\not\Subset \mathcal{W}$.
\end{defi}

\begin{thm}\label{cd}
For a skill multimap $(Q, S, \mu)$ with each $\mu(r)$ being finite, then the delineate knowledge structure $(Q, \mathscr{H})$ is completely discriminative iff for any distinct $h, q$ in $Q$, there exist $C_{h}\in\mu_{M}(h)$ and $C_{q}\in\mu_{M}(q)$ such that, for any $g\in Q$, at most one of $C_{h}\Subset \mu_{M}(g)$ and $C_{q}\Subset \mu_{M}(g)$ holds.
\end{thm}

\begin{proof}
Necessity. Assume that $(Q, \mathscr{H})$ is completely discriminative. Then for any distinct $h, q$ in $Q$, there exist $C_{h}\in\mu_{M}(h)$ and $C_{q}\in\mu_{M}(q)$ such that $p(C_{h})\cap p(C_{q})=\emptyset$. For any $g\in Q$, without loss of generality, suppose that $C_{h}\Subset \mu_{M}(g)$, then $g\in p(C_{h})$. Since  $p(C_{h})\cap p(C_{q})=\emptyset$, it follows that $g\not\in p(C_{q})$, then $C\setminus C_{q}\neq\emptyset$ for any $C\in\mu_{M}(g)$. Hence $C_{q}\not\Subset \mu_{M}(g)$.

Sufficiency.  For any distinct $h, q$ in $Q$, it follows from the assumption that there exist $C_{h}\in\mu_{M}(h)$ and $C_{q}\in\mu_{M}(q)$ such that, for any $g\in Q$, at most one of $C_{h}\Subset \mu_{M}(g)$ and $C_{q}\Subset \mu_{M}(g)$ holds. Then it is easily verified that $p(C_{h})\cap p(C_{q})=\emptyset$. Therefore, $(Q, \mathscr{H})$ is completely discriminative.
\end{proof}

The following corollary is easily checked by Theorems~\ref{ttttt2} and~\ref{cd}.

\begin{cor}\label{cd1}
For a skill multimap $(Q, S, \mu)$ with each $\mu(r)$ being finite, then the delineate knowledge structure $(Q, \mathscr{H})$ is a completely discriminative knowledge space iff the following two conditions hold:

\smallskip
(1) For any $H\subseteq Q$, $H\in\mathscr{H}$ iff there exists $\mathcal{P}_{H}\subseteq\bigcup_{t\in Q}\mu_{M}(t)$ such that $H=\bigcup_{D\in\mathcal{P}_{H}}p(D)$;

\smallskip
(2) For any distinct $h, q$ in $Q$, there exist $C_{h}\in\mu_{M}(h)$ and $C_{q}\in\mu_{M}(q)$ such that, for any $g\in Q$, at most one of $C_{h}\Subset \mu_{M}(g)$ and $C_{q}\Subset \mu_{M}(g)$ holds.
\end{cor}

The following Theorems~\ref{t17} and~\ref{t18} show that the language of the regularity of pre-topology in knowledge spaces.

\begin{thm}\label{t17}
Let ($Q, \mathscr{H}$) be a knowledge space, and let $(Q, \tau_{\mathscr{H}})$ be regular. For each $t\in Q$ and $H\in \mathscr{H}_{t}$, we have $Q\setminus H\in\mathscr{H}$ or $H$ is not an atom at $t$.
\end{thm}

\begin{proof}
Take an arbitrary $t\in Q$, and any $H\in \mathscr{H}_{t}$. Assume that $Q\setminus H\not\in\mathscr{H}$, then $H$ is not closed in $(Q, \tau_{\mathscr{H}})$. Since $(Q, \tau_{\mathscr{H}})$ is regular and $t\not\in Q\setminus H$, there are $L\in\mathscr{H}$ and $K\in\mathscr{H}$ with $t\in L$, $Q\setminus H\subseteq K$ and $L\cap K=\emptyset$. Because $Q\setminus K\not\in\mathscr{H}$, it follows that $K\cap H\neq\emptyset$; moreover, it is obvious that $L\subseteq Q\setminus K\subseteq H$, then $L\neq H$ since $Q\setminus H\not\in\mathscr{H}$, hence $H$ is not an atom at $t$.
\end{proof}

\begin{cor}
Let ($Q, \mathscr{H}$) be a knowledge space with $(Q, \tau_{\mathscr{H}})$ being regular. If there exists an atom $K$ at some $t\in Q$, then $(Q, \tau_{\mathscr{H}})$ is not a connected space. In particular, if $Q$ is finite and $(Q, \tau_{\mathscr{H}})$ is regular, then $(Q, \tau_{\mathscr{H}})$ is not a connected space.
\end{cor}

By a similar proof of Theorem~\ref{t17}, the following theorem holds.

\begin{thm}\label{t18}
Let ($Q, \mathscr{H}$) be a knowledge space with $(Q, \tau_{\mathscr{H}})$ being regular. For each $K\in \mathscr{H}$ and $t\in Q$, if $t\in K^{\mathcal{O}}$, then $Q\setminus (K\cup\{t\})\in\mathscr{H}$ or $K\cup\{t\}$ is not an atom at $t$.
\end{thm}

\begin{rem}
It is an interesting topic to find the applications of axioms of separation of Tychonoff and normality in knowledge spaces.
\end{rem}

\maketitle
\subsection{Alexandroff spaces and quasi ordinal spaces}
\
\newline
\indent It is well known that a topological space is called {\it Alexandroff spaces} provided the intersection of arbitrary many open sets is open. From Theorem~\ref{t19}, we see that ordinal spaces is equivalent to $T_{0}$-Alexandroff spaces.

\begin{thm}\label{t0}
The knowledge space ($Q, \mathscr{H}$) is an ordinal space iff ($Q, \tau_{\mathscr{H}}$) is a $T_{0}$-Alexandroff space.
\end{thm}

\begin{proof}
Since each Alexandroff space is equivalent to quasi ordinal space, it following Theorem~\ref{t19} that each ordinal space is equivalent to a $T_{0}$-Alexandroff space.
\end{proof}

The following theorem holds by \cite[Theorem 3.8.3]{falmagne2011learning} and Theorem~\ref{t0}.

\begin{thm}
There exists a bijection between the collection of all quasi orders $\mathcal{Q}$ on $Q$ and the collection of all Alexandroff spaces $\mathscr{T}$ on $Q$. Under this bijection, $T_{0}$-Alexandroff spaces are mapped onto partial orders.
\end{thm}

If a relation on a set $Z$ is reflexive and transitive, then we say this relation is a {\it quasi order} $Z$, and if $Z$ is equipped with a quasi order is called {\it quasi ordered}.
Given a quasi-order $\verb"Q"$ construct the Alexandroff space $Q(\verb"Q")$ as the set $\verb"Q"$ with the topology generated by $\{]\leftarrow, q]: q\in Q\}$. Conversely, given an Alexandroff space $(Q, \tau)$, we construct a quasi-order $\verb"Q"(Q)$ as the set with the order $x\preceq y$ iff $x\in\bigcap\tau(y)$.

\begin{prop}
Let ($Q, \mathscr{F}_{1})$ and ($Q, \mathscr{F}_{2})$ be two pre-topologies on $Q$ such that ($Q, \mathscr{F}_{1})$ and ($Q, \mathscr{F}_{2})$ are two Alexandroff space. Then $\mathscr{F}_{1}=\mathscr{F}_{2}$ iff, for any $p, q\in Q$, $\mathscr{F}_{1}(p)\subseteq \mathscr{F}_{1}(q)\Leftrightarrow \mathscr{F}_{2}(p)\subseteq \mathscr{F}_{2}(q)$.
\end{prop}

\begin{proof}
The necessity is obvious. In order to prove the converse implication, suppose that $O\in \mathscr{F}_{1}$. Hence $O\subseteq\bigcup_{p\in O}(\bigcap\mathscr{F}_{2}(p))\in\mathscr{F}_{2}.$ Put $O^{\prime}=\bigcup_{p\in O}(\bigcap\mathscr{F}_{2}(p))$. We conclude that $O=O^{\prime}$. Indeed, take any $u\in O^{\prime}$. Then there exists $v\in O$ such that $u\in\bigcap\mathscr{F}_{2}(v)$. Hence $\mathscr{F}_{2}(v)\subseteq\mathscr{F}_{2}(u)$, then $\mathscr{F}_{1}(v)\subseteq\mathscr{F}_{1}(u)$. We obtain $u\in\bigcap\mathscr{F}_{1}(v)$, which implies that $u\in O$. This gives $O^{\prime}\subseteq O$, hence $O=O^{\prime}$. We conclude that $\mathscr{F}_{1}\subseteq\mathscr{F}_{2}$, and by symmetry, $\mathscr{F}_{1}=\mathscr{F}_{2}$.
\end{proof}

\begin{prop}
If ($Q, \mathscr{H}$) is a knowledge structure, then the following are equivalent:

\begin{enumerate}
\smallskip
\item ($Q, \mathscr{H}$) is a finite ordinal space;

\smallskip
\item ($Q, \mathscr{H}$) is a learning space and $\bigcap\mathscr{H}_{p}\in\mathscr{H}$ for each $p\in Q$;

\smallskip
\item ($Q, \tau_{\mathscr{H}}$) is a finite $T_{0}$-Alexandroff space.
\end{enumerate}
\end{prop}

\begin{proof}
By Theorem~\ref{t0}, we have (1) $\Leftrightarrow$ (3). It follows from \cite[Theorem 3.8.7]{falmagne2011learning} that (1) $\Rightarrow$ (2). Hence it suffices to prove (2) $\Rightarrow$ (3). Since ($Q, \mathscr{H}$) is a learning space, it follows that ($Q, \mathscr{H}$) is finite and discriminative, hence ($Q, \tau_{\mathscr{H}}$) is $T_{0}$. Suppose that $H, L\in\mathscr{H}$. If $H\cap L$ is empty, then $H\cap L\in\mathscr{H}$. Assume that $H\cap L\neq\emptyset$. For each $p\in H\cap L$, since $\bigcap\mathscr{H}_{p}\in\mathscr{H}$, we have $\bigcap\mathscr{H}_{p}\subset H$ and $\bigcap\mathscr{H}_{p}\subset L$, hence $\bigcap\mathscr{H}_{p}\subset H\cap L$. Therefore, $H\cap L=\bigcup_{p\in H\cap L}\bigcap\mathscr{H}_{p},$ which shows that $H\cap L\in\mathscr{H}$. Thus ($Q, \tau_{\mathscr{H}}$) is a finite $T_{0}$-Alexandroff space.
\end{proof}

Let $\mathscr{H}$ be a family of subsets of a finite set $Q$ such that $\mathscr{H}$ is
closed under union. The family $\mathscr{H}$ is called an {\it antimatroid} provided $Q=\bigcup\mathscr{H}$ and $\mathscr{H}$ satisfies the condition: for each non-empty element $H$ of $\mathscr{H}$, there is $t\in H$ such that $H\setminus\{t\}\in\mathscr{H}$.

\begin{lem}\label{l3}
Any finite $T_{0}$-Alexandroff space $X$ is antimatroid.
\end{lem}

\begin{proof}
For each point $x\in X$, let $V_{x}$ be a minimal open set containing the point $x$. Let $U$ be any non-empty open set in $X$. Take an arbitrary point $y_{0}\in U$. If $U$ is a single set, then it is obviously true. Thus we assume that $|U|\geq 2$. If $\bigcup_{y\in U\setminus\{y_{0}\}}V_{y}=U\setminus\{y_{0}\}$, then $U\setminus\{y_{0}\}$ is obviously open in $X$, hence the assertion of the Lemma is true. Therefore, now we assume that $\bigcup_{y\in U\setminus\{y_{0}\}}V_{y}\neq U\setminus\{y_{0}\}$, then there is $y_{1}\in U\setminus\{y_{0}\}$ with $y_{1}\not\in V_{y_{0}}$ since $X$ is $T_{0}$. If $U\setminus V_{y_{0}}=\{y_{1}\}$, then the assertion of the Lemma is true. Otherwise, $|U\setminus V_{y_{0}}|\geq 2$. If $V_{y_{0}}\cup\bigcup_{y\in U\setminus(V_{y_{0}}\cup\{y_{1}\})}V_{y}=U\setminus\{y_{1}\}$, then $U\setminus\{y_{1}\}$ is obviously open in $X$, hence the assertion of the Lemma is true. Therefore, we may assume that $V_{y_{0}}\cup\bigcup_{y\in U\setminus(V_{y_{0}}\cup\{y_{1}\})}V_{y}\neq U\setminus\{y_{1}\}$, then there exists $y_{2}\in U\setminus(V_{y_{0}}\cup\{y_{1}\})$ with $y_{2}\not\in V_{y_{1}}\cup V_{y_{0}}$ since $X$ is $T_{0}$. Since $U$ is finite, by induction, it follows that there exists a point $t\in U$ so that $U\setminus\{t\}$ is open in $X$.
\end{proof}

\begin{thm}
Any finite $T_{0}$-Alexandroff space $X$ is tight 1-connected.
\end{thm}

\begin{proof}
Since $X$ is a finite $T_{0}$-Alexandroff space, we denote the minimal open neighborhood of $x$ by $W_{x}$. By Theorem~\ref{t12}, it only need to prove that $(O\triangle W)\cap O^{\mathcal{LC}}\neq\emptyset$ for any open sets $O$ and $W$ with $O\neq W$. We divide the rest proof into the following two cases.

 \smallskip
 {\bf Case 1:} $O\setminus W\neq\emptyset$.

 \smallskip
It follow from Lemma~\ref{l3} and \cite[Lemma 6]{2009Danilov} that we can take a point $z\in O\setminus W$ so that $O\setminus\{z\}$ is open in $X$. Therefore, $z\in(O\triangle W)\cap O^{\mathcal{LC}}$.

 \smallskip
 {\bf Case 2:} $O\setminus W=\emptyset$.

 \smallskip
Indeed, we shall use an induction on the size of the set $W\setminus O$. If $W\setminus O$ is a single set $\{x\}$, then $x\in(O\triangle W)\cap O^{\mathcal{LC}}$. Suppose that there is a point $x\in W\setminus O$ such that $O\setminus\cup\{x\}$ is open if $|W\setminus O|=n$ for each $n\geq 2$. Now assume that $|W\setminus O|=n+1$. Then it follow from Lemma~\ref{l3} and \cite[Lemma 6]{2009Danilov} that there exists a point $z\in W\setminus O$ such that $W(z)=W\setminus\{z\}$ is open in $X$. Then $|W(z)\setminus O|=n$. By our assumption, there exists $y\in W(z)\setminus O$ such that $O\cup\{y\}$ is open in $X$.
\end{proof}

Let $\mathscr{F}$ be a family of subsets of a finite set $Z$ . We say that $\mathscr{F}$ is {\it well-graded} provided, for any
$K, H\in\mathscr{F}$ with $L\neq H$, there exists a finite sequence of states $L=H_{0}, H_{1}, \ldots, H_{m}=H$ so
that $d(H_{i-1}, H_{i})=1$ for $1\leq i\leq m$, where $m=d(L, H)$.

\begin{cor}\label{c2}
Each finite ordinal space $(Q, \mathscr{H})$ is well-graded.
\end{cor}

\begin{rem}
There is a finite ordinal space $(Q, \mathscr{H})$ such that $(Q, \tau_{\mathscr{H}})$ is not connected. Indeed, let $Q=\{z_{1}, z_{2}, z_{3}, z_{4}, z_{5}, z_{6}\}$ endowed with the following knowledge structure
$$\mathscr{H}=\{\emptyset, \{z_{4}\}, \{z_{1}, z_{3}\}, \{z_{5}, z_{6}\}, \{z_{1}, z_{3}, z_{4}\}, \{z_{1}, z_{2}, z_{3}\}, \{z_{4}, z_{5}, z_{6}\},$$$$\{z_{1}, z_{3}, z_{5}, z_{6}\}, \{z_{1}, z_{2}, z_{3}, z_{4}\}, Q\}.$$
Then $(Q, \mathscr{H})$ is a finite ordinal space, which is not connected since $\{z_{1}, z_{2}, z_{3}\}$ is open and closed in $(Q, \tau_{\mathscr{H}}).$ However, it follow from \cite[Theorem 2.7]{Arenas1999Alexandroff} that we have the following Theorem~\ref{t20}. First, we give the following proposition.
\end{rem}

\begin{prop}
Assume that $(Q, \mathscr{H})$ is a knowledge space. Then $(Q, \mathscr{H})$ is a quasi ordinal space iff each point in $Q$ has a unique minimal state $M(t)$ containing the point $t$.
\end{prop}

\begin{proof}
Necessity. Let $(Q, \mathscr{H})$ be a quasi ordinal space with $t\in Q$. Put $$\mathcal{M}(t)=\{G: t\in G\in\mathscr{H}\}$$and $M(t)=\bigcap\mathcal{M}(t)$. Since $(Q, \mathscr{H})$ is a quasi ordinal space, it follows that $M(t)\in\mathscr{H}$. From the definition of $M(t)$, it is obvious that $M(t)$ is a unique minimal state containing the point $t$.

Sufficiency. Assume each $t\in Q$ has a unique minimal state $M(t)$ containing the point $t$. Consider an arbitrary intersection of states, $H=\bigcap_{\beta\in A}H_{\beta}$, where every $H_{\beta}\in\mathscr{H}$. If $H=\emptyset$, then $H\in\mathscr{H}$ and we are done. If $H\neq\emptyset$, then take any $t\in H$ and we have $t\in H_{\beta}$ for each $\beta\in A$, hence $M(t)\subseteq H_{\beta}$ for each $\beta\in A$ since $M(t)$ is the unique state containing the point $t$. Therefore, $M(t)\subseteq H$ for each $\beta\in A$, thus $H=\bigcup_{t\in H}M(t)\in \mathscr{H}$ since $(Q, \mathscr{H})$ is a knowledge space. Therefore, $(Q, \mathscr{H})$ is a quasi ordinal space.
\end{proof}

\begin{thm}\label{t20}
If $(Q, \mathscr{H})$ is an ordinal space, then the following conditions are equivalent:

\smallskip
(a) $(Q, \tau_{\mathscr{H}})$ is connected;

\smallskip
(b) $(Q, \tau_{\mathscr{H}})$ is chain connected;

\smallskip
(c) $(Q, \tau_{\mathscr{H}})$ is pathwise connected;

\smallskip
(d) for any $r, t\in Q$, there exists a finite sequence $q_{0}, \ldots, q_{n+1}\in Q$ such that $r=q_{0}, q_{n+1}=t$ and $M(q_{i})\cap M(q_{j})\neq\emptyset$ if $|i-j|\leq 1$.
\end{thm}

\begin{thm}
If $(Q, \mathscr{H})$ is a bi-discriminative quasi ordinal space, then $\mathscr{H}=2^{Q}$.
\end{thm}

\begin{proof}
It only need to prove that $\{t\}\in \mathscr{H}$ for each $t\in Q$. Indeed, take any $t\in Q$. Obviously, $(Q, \mathscr{H})$ is a $T_{1}$ Alexandroff space. Therefore, it follows that $$\{t\}=\bigcap\{H\in\mathscr{H}: t\in H\}\in\mathscr{H}.$$
\end{proof}

From \cite[Theorem 2.9]{Arenas1999Alexandroff}, the following two theorems hold.

\begin{thm}
If $(Q, \mathscr{H})$ is a quasi ordinal space, then $Q\setminus M(t)\in \mathscr{H}$ for each $t\in Q$ iff the pre-topological space $(Q, \tau_{\mathscr{H}})$ is regular, where $M(t)$ is the unique minimal state containing the point $t$.
\end{thm}

\begin{thm}
Let $(Q, \mathscr{H})$ be a quasi ordinal space. If $(Q, \tau_{\mathscr{H}})$ is a regular pre-topological space with a countable dense subset, then $(Q, \tau_{\mathscr{H}})$ is normal.
\end{thm}

\maketitle
\subsection{Primary items in knowledge spaces}
\
\newline
\indent In this section, we discuss the density of pre-topological spaces with the application in knowledge spaces. By Theorem~\ref{t13}, it is easily verified the following theorem.

\begin{thm}
The inequality $|Q|\leq |\mathscr{H}|$ holds for each discriminative knowledge space $(Q, \mathscr{H})$.
\end{thm}

We say that a pairwise disjoint family consisting of non-empty open subsets of a pre-topological space $(Z, \tau)$ is called a {\it cellular family}. The {\it cellularity} of $Z$ is defined as follows: $$c(Z)=\sup\{|\mathscr{V}|: \mathscr{V}\ \mbox{ is a cellular family in}\ Z\}.$$ Here, the cellularity of a pre-topological space maybe finite.

\begin{lem}\label{l4}
Let $Z$ be a pre-topological space such that $c(Z)=\kappa$. If $\mathscr{V}$ is an open cover of $Z$, then there is a subcollection $\mathscr{W}$ of $\mathscr{V}$ with $|\mathscr{W}|\leq\kappa$ and $\overline{\bigcup\mathscr{W}}=Z$.
\end{lem}

\begin{proof}
Let $\mathscr{O}$ be the family of all non-empty open sets in $Z$ that are subsets of some element of $\mathscr{V}$. It follows from Zorn's lemma that we can take a maximal cellular family $\mathscr{O}^{\prime}\subset\mathscr{O}$. Hence $|\mathscr{O}^{\prime}|\leq c(X)=\kappa$, and $Z=\overline{\bigcup\mathscr{O}^{\prime}}$ by the maximality of $\mathscr{O}^{\prime}$. For each $O\in\mathscr{O}^{\prime}$, fix a $W_{O}\in \mathscr{V}$ such that $O\subset  W_{O}$. Put $\mathscr{W}=\{W_{O}: O\in\mathscr{O}^{\prime}\}$. Then $\overline{\bigcup\mathscr{W}}=Z$.
\end{proof}

The following corollary holds by Lemma~\ref{l4}.

\begin{cor}
Let $\mathcal{B}$ be a atom pre-base of a finite pre-topological space $Z$. Then there is a maximally disjoint subfamily $\mathcal{B}^{\prime}$ of $\mathcal{B}$ such that $Z=\overline{\mathcal{B}^{\prime}}$.
\end{cor}

Let $(Q, \mathscr{H})$ be a knowledge space. If $D$ is a dense subset of $(Q, \tau_{\mathscr{H}})$, we say that $D$ is the {\it primary items} of $Q$. In other work, one person in order to master an item, she or he must master some items of $D$.

\begin{defi}\label{defi1}
Let $Z$ be a pre-topological space. The infimum of the set $$\{|D|: D\subseteq Z, \overline{D}=Z\}$$ is called the {\it density} $d(Z)$ of $Z$.
\end{defi}

For a knowledge space $(Q, \mathscr{H})$, we say that $A$ is an {\it atom primary items} of $(Q, \mathscr{H})$ if $A$ is dense and $|A|=d(Q)$. Moreover, the atom primary items of a knowledge space is not unique. Indeed, the sets $\{z_{1}, z_{2}, z_{3}\}$,  $\{z_{1}, z_{2}, z_{4}\}$ and $\{z_{2}, z_{3}, z_{4}\}$ are all atom primary items of the pre-topological space $(Z, \tau)$ in Example~\ref{e1}. Clearly, we have the following proposition.

\begin{prop}
For any knowledge space $(Q, \mathscr{H})$, if $A$ is an atom primary items of $(Q, \mathscr{H})$, then $|A|\leq \min\{|Q|, w(Q, \tau_{\mathscr{H}})\}$.
\end{prop}

From Lemma~\ref{l4}, the following proposition holds.

\begin{prop}\label{p7}
For any knowledge space $(Q, \mathscr{H})$, if $A$ is an atom primary items of $(Q, \mathscr{H})$, then $|A|\geq c(Q)$.
\end{prop}

The following theorem gives a complement for Example~\ref{e2}.

\begin{thm}\label{th1}
Let $(Q, \mathscr{H})$ be a knowledge space, and let $D$ be a dense subset of $(Q, \tau_{\mathscr{H}})$. If $(Q, \tau_{\mathscr{H}})$ is Hausdorff and $\overline{H\cap D}=\overline{H}$ for each $H\in \mathscr{H}$, then $|Q|\leq 2^{2^{|D|}}$.
\end{thm}

\begin{proof}
For every $t\in Q$, the family $\mathcal{N}(t)=\{U\cap A: U\in\mathscr{H}_{t}\}$ of subsets of $D$. Since $(Q, \tau_{\mathscr{H}})$ is Hausdorff and $\overline{L\cap D}=\overline{L}$ for any $L\in \mathscr{H}$, we conclude that the intersection of the closures of all members of $\mathcal{N}(t)$ equals $\{t\}$; thus $\mathcal{N}(t)\neq\mathcal{N}(g)$ for $g\neq q$. Clearly, the number of all distinct families $\mathcal{N}(t)$ is not larger than $2^{2^{|D|}}$, hence $|Q|\leq 2^{2^{|D|}}$.
\end{proof}

\begin{rem}
The Hausdorff pre-topological space $(Z, \tau)$ in Example~\ref{e1}, which has an atom primary items $D=\{z_{1}, z_{2}, z_{3}\}$ such that $|Z|\leq 2^{2^{3}}$ and $|D|=3<4$. However, $$\{z_{2}\}=\overline{\{z_{2}\}}=\overline{\{z_{2}, z_{4}\}\cap \{z_{1}, z_{2}, z_{3}\}}\neq \{z_{2}, z_{4}\}=\overline{\{z_{2}, z_{4}\}}.$$ Hence the condition ``$\overline{H\cap D}=\overline{H}$ for each $H\in \mathscr{H}$'' is sufficient and not necessary.
\end{rem}

\begin{lem}\label{l5}
Let $(Q, \mathscr{H})$ be a finite knowledge space, and let $\mathscr{H}^{\prime}$ be a subfamily of $\mathscr{H}$. Put $m=\max\{|\mathscr{H}^{\prime}_{q}|: q\in Q\}$. For any distinct $r, t\in Q$, if $|\mathscr{H}^{\prime}_{r}|=|\mathscr{H}^{\prime}_{t}|=m$, then $\bigcap\mathscr{H}^{\prime}_{r}=\bigcap\mathscr{H}^{\prime}_{t}$ or $(\bigcap\mathscr{H}^{\prime}_{r})\cap(\bigcap\mathscr{H}^{\prime}_{t})=\emptyset$.
\end{lem}

\begin{proof}
Assume that $\bigcap\mathscr{H}^{\prime}_{r}\neq\bigcap\mathscr{H}^{\prime}_{t}$. Then, without loss of generality, we assume that there is $H\in \mathscr{H}^{\prime}_{t}$ with $r\not\in H$. Now assume that $(\bigcap\mathscr{H}^{\prime}_{r})\cap(\bigcap\mathscr{H}^{\prime}_{t})\neq\emptyset$. Then there is $q\in (\bigcap\mathscr{H}^{\prime}_{r})\cap(\bigcap\mathscr{H}^{\prime}_{t})$, hence $q\in H$ and $q\in L$ for any $L\in\mathscr{H}^{\prime}_{r}$, which implies that $\mathscr{H}^{\prime}_{q}=m+1$, this is a contradiction. Therefore, $(\bigcap\mathscr{H}^{\prime}_{r})\cap(\bigcap\mathscr{H}^{\prime}_{t})=\emptyset$.
\end{proof}

Now we give a method to construct a primary items for a finite pre-topological space as follows.

Let $\mathscr{B}$ be an atom pre-base for a finite pre-topological space $(Q, \mathscr{H})$. Now we define a dense subset $A$ of $(Q, \tau_{\mathscr{H}})$, by induction, as follows. First, let $\mathscr{B}(1)=\mathscr{B}$, and take an arbitrary $p_{1}\in Q$ such that $|\mathscr{B}(1)_{p_{1}}|=\max\{|\mathscr{B}(1)_{q}|: q\in Q\}$. Put $A_{1}=\{p_{1}\}.$ If $\mathscr{B}\setminus\mathscr{B}(1)_{p_{1}}=\emptyset$, then we stop our induction, and put $A=A_{1}$. If $\mathscr{B}\setminus\mathscr{B}(1)_{p_{1}}\neq\emptyset$, then put $\mathscr{B}(2)=\mathscr{B}\setminus\mathscr{B}(1)_{p_{1}}$. Then take an arbitrary $p_{2}\in Q$ such that
$|\mathscr{B}(2)_{p_{2}}|=\max\{|\mathscr{B}(2)_{q}|: q\in Q\}$. Put $A_{2}=A_{1}\cup\{ p_{2}\}$. Assume that we have define a subset $A_{k}=\{p_{i}: i\leq k\} (k\geq 2)$ such that  $|\mathscr{B}(i)_{p_{i}}|=\max\{|\mathscr{B}(i)_{q}|: q\in Q\}$ for any $1<i\leq k$, where $\mathscr{B}(i)=\mathscr{B}\setminus \bigcup_{j<i}\mathscr{B}(j)_{p_{j}}$. If $\mathscr{B}\setminus\bigcup_{i\leq k}\mathscr{B}(i)_{p_{i}}=\emptyset$, then we stop our induction, and put $A=A_{k}$; otherwise, let $\mathscr{B}(k+1)=\mathscr{B}\setminus\bigcup_{i\leq k}\mathscr{B}(i)_{p_{i}}$. Then take an arbitrary $p_{k+1}\in Q$ such that
$|\mathscr{B}(k+1)_{p_{k+1}}|=\max\{|\mathscr{B}(k)_{q}|: q\in Q\}$. Then put $A_{k+1}=A_{k}\cup\{p_{k+1}\}.$ Since $Q$ is finite, there exists a minimum $n\in\mathbb{N}$ such that $\mathscr{B}\setminus\bigcup_{i\leq n}\mathscr{B}(i)_{p_{i}}=\emptyset$. Moreover, from Lemma~\ref{l5}, it follows that, for any $j\leq n$, if $|\mathscr{B}(j)|=1$ we may assume that $p_{j}\not\in\bigcup(\bigcup_{i<j}\mathscr{B}(i)_{p_{i}})$. Then we put $A=A_{n}$. Clearly, $A$ is dense in $(Q, \mathscr{H})$.

Generally, $A$ is not an atom primary items of $(Q, \mathscr{H})$. We must delete some surplus points of $A$ such that it is an atom primary items of $(Q, \mathscr{H})$.

If $B\cap (A\setminus\{p_{1}\})\neq\emptyset$ for any $B\in\mathscr{B}(1)_{p_{1}}$, then put $B_{1}=A\setminus \{p_{1}\}$; otherwise, put $B_{1}=A$. If $p_{1}\in B_{1}$ and $B\cap (B_{1}\setminus\{p_{2}\})\neq\emptyset$ for any $B\in\mathscr{B}(2)_{p_{2}}$, then put $B_{2}=B_{1}\setminus \{p_{2}\}$; if $p_{1}\in B_{1}$ and $B\cap (B_{1}\setminus\{p_{2}\})=\emptyset$ for some $B\in\mathscr{B}(2)_{p_{2}}$, then put $B_{2}=B_{1}$; if $p_{1}\not\in B_{1}$, $B\cap (B_{1}\setminus\{p_{2}\})\neq\emptyset$ for any $B\in\mathscr{B}(2)_{p_{2}}$, and $O\cap (B_{1}\setminus\{p_{2}\})\neq\emptyset$ whenever $p_{2}\in O\in \mathscr{B}(1)_{p_{1}}$, then put $B_{2}=B_{1}\setminus \{p_{2}\}$; otherwise, put $B_{2}=B_{1}$. By induction, we can obtain the subsets $B_{1}, \ldots, B_{n}$ of $A$ such that $B_{1}\supset\ldots\supset B_{n}$ and $B_{n}$ is dense in $(Q, \mathscr{H})$. Clearly, we have $p_{n}\in B_{n}$. We claim that $D=B_{n}$ is an atom primary items of $(Q, \mathscr{H})$, see the following theorem.

\begin{thm}
The set $D$ is dense in $(Q, \tau_{\mathscr{H}})$ and $|D|=d(Q)$, that is, $D$ is an atom primary items of $(Q, \mathscr{H})$.
\end{thm}

\begin{proof}
From our construction of $D$ above, it follows that $D$ is dense in $(Q, \tau_{\mathscr{H}})$. Hence it only need to prove that $\overline{C}\neq Q$ for ecah subset $C$ of $Q$ with $|C|<|D|$. Indeed, we shall prove by induction on the cardinality of atom pre-base $\mathscr{B}$.

Clearly, if $|\mathscr{B}|=1$, then $|D|=1$, hence $\overline{C}=\emptyset$. If $|\mathscr{B}|=2$, it suffice to verify the case of $D=\{p_{1}, p_{2}\}$. Let $C=\{q_{0}\}$. By our construction, if $C$ is dense in $(Q, \tau_{\mathscr{H}})$, then $|\mathscr{B}_{q_{0}}|>|\mathscr{B}_{p_{1}}|$, which is a contradiction. Assume that $\overline{C}\neq Q$ for any subset $C$ of $Q$ with $|C|<|D|$ when $|\mathscr{B}|\leq k$. Now assume that $|\mathscr{B}|=k+1$. From our construction, we conclude that $|A\cap B|=1$ for each $B\in\mathscr{B}(n)_{p_{n}}$. We divide the rest proof into the following two cases.

\smallskip
{\bf Case 1:} $|\mathscr{B}(n)_{p_{n}}|\geq 2$.

\smallskip
Let $B^{\prime}=\bigcup\mathscr{B}(n)_{p_{n}}$, and put $\mathscr{B}^{\prime}=(\mathscr{B}\setminus\mathscr{B}(n)_{p_{n}})\cup\{B^{\prime}\}$. Let $(Q, \mathscr{F})$ be the knowledge space generated by $\mathscr{B}^{\prime}$. Clearly, $\tau_{\mathscr{F}}$ is coarser than $\tau_{\mathscr{H}}$, hence $D$ is dense in $\tau_{\mathscr{F}}$. Since $|\mathscr{B}^{\prime}|\leq k$, from our assumption it follows that $\overline{F}^{\tau_{\mathscr{F}}}\neq Q$ for each subset $F$ of $Q$ with $|F|<|D|$, hence $\overline{F}^{\tau_{\mathscr{H}}}\neq Q$ for each subset $F$ of $Q$ with $|F|<|D|$.

\smallskip
{\bf Case 2:} $|\mathscr{B}(n)_{p_{n}}|=1$.

\smallskip
If $|\mathscr{B}(i)_{p_{i}}|=1$ for each $i\leq n$, then it is obvious. Otherwise, there exists a maximal $m\leq n$ such that $|\mathscr{B}(i)_{p_{i}}|=1$ for any $i>m$ and $|\mathscr{B}(m)_{p_{m}}|\geq 2$. Then put $B^{\prime\prime}=\bigcup\mathscr{B}(m)_{p_{m}}$. Let $(Q, \mathscr{L})$ be the knowledge space which is generated by the family $\mathscr{B}^{\prime\prime}=(\mathscr{B}\setminus \mathscr{B}(m)_{p_{m}})\cup \{B^{\prime\prime}\}$. Then $\tau_{\mathscr{L}}$ is coarser than $\tau_{\mathscr{H}}$, hence $D$ is dense in $(Q, \tau_{\mathscr{L}})$. Since $|\mathscr{B}^{\prime\prime}|\leq k$, from our assumption it follows that $\overline{F}^{\tau_{\mathscr{L}}}\neq Q$ for each subset $F$ of $Q$ with $|F|<|D|$, hence $\overline{F}^{\tau_{\mathscr{H}}}\neq Q$ for each subset $F$ of $Q$ with $|F|<|D|$.
\end{proof}

Finally we give an algorithm for constructing the set of atom primary items for a finite knowledge spaces.

{\bf Sketch of Algorithm.} Let $Q=\{q_{1}, \ldots, q_{m}\}$ , and list the atom pre-base $\mathcal{B}$ as $B_{1}, \ldots, B_{n}$. Form an $m\times n$ array $T=(T_{ij})$ with the rows and columns representing the atom pre-base $\mathcal{B}$ and the elements of $Q$ respectively; thus, both the rows and the columns are indexed from 1 to $m$ and 1 to $n$ respectively. Initially, set $T_{ij}$ to $1$ if $q_{i}\in B_{j}$; otherwise, set $T_{ij}$ to 0.

(1) First, take an arbitrary $i_{1}\in\{1, \ldots, m\}$ such that $\sum_{j=1}^{n}T_{i_{1}j}=\max\{\sum_{j=1}^{n}T_{ij}: i\in\{1, \ldots, m\}\}$, then swap the $i_{1}$-th row of the matrix with the first row, and swap all the columns $N_{1}=\{j: T_{i_{1}j}=1, j=1, \ldots, n\}$ with the first $n_{1}$ columns, where $n_{1}=|N_{1}|$. If $n_{1}=n$, then we terminates this step, and put $A=\{p_{i_{1}}\}$; otherwise, denote the new matrix by $(T_{ij}^{(1)})_{m\times n}$.

(2) For any $k\geq 2$, take an arbitrary $i_{k}\in\{k, \ldots, m\}$ such that $$\sum_{j=(\sum_{j=1}^{k-1}n_{j})+1}^{n}T_{i_{k}j}^{(k-1)}=\max\{\sum_{j=(\sum_{j=1}^{k-1}n_{j})+1}^{n}T_{ij}^{(k-1)}: i\in\{k, \ldots, m\}\},$$ then swap the $i_{k}$-th row of the matrix with the $k$-th row, and swap all the columns $N_{k}=\{j: T_{i_{k}j}=1, j=(\sum_{j=1}^{k-1}n_{j})+1, \ldots, n\}$ with the $j=(\sum_{j=1}^{k-1}n_{j})+1$-th column to $j=(\sum_{j=1}^{k}n_{j})$-th column, where $n_{k}=|N_{k}|$. If $n_{k}=1$, we may require the $i_{k}$ satisfying that $T_{i_{k}j}^{(k)}=0$ for any $j\leq \sum_{j=1}^{k-1}n_{j}$. If $\sum_{j=1}^{k}n_{j}=n$, then we terminates this step, and put $A=\{p_{i_{1}}, \ldots, p_{i_{k}}\}$; otherwise, denote the new matrix by $(T_{ij}^{(k)})_{m\times n}$.

(3) The set $A=\{p_{i_{1}}, \ldots, p_{i_{N}}\}$ and the sub-matrix $(T_{ij}^{(N)})_{N\times M}$ of $(T_{ij}^{(N)})_{m\times n}$ obtained after step $N$ are the desired set and the desired matrix respectively, where $M=\sum_{i}^{N}n_{i}$.

(4) Then initialize $D$ to $A$ again.

(5) For each $l\leq N$, check whether, there exists $\sum_{k=1}^{l-1}n_{k}<j\leq\sum_{k=1}^{l}n_{k}$ such that $T_{i_{l}j}^{(N)}=0$ for any $i_{l}>l$; if the condition holds, the set $D$ is invariant. If the condition does not hold, check whether there exists $t<l$ such that $p_{i_{t}}\not\in D$ and $\sum_{j=a_{t-1}+1}^{a_{t}}T_{i_{l}j}^{(N)}=1$, where $a_{t}=\sum_{i=1}^{t}n_{i}$ and $n_{0}=0$; otherwise, delete $p_{i_{l}}$ from $D$. (This terminates step $N$.)

The set $D$ obtained after step $N$ is the set of atom primary items for a finite knowledge spaces $(Q, \mathscr{H})$.

\begin{ex}
Let a knowledge space $(Q, \mathscr{H})$ have an atom pre-base $$\{\{z_{1}\}, \{z_{2}\}, \{z_{1}, z_{3}\}, \{z_{2}, z_{3}, z_{4}\}, \{z_{1}, z_{3}, z_{4}, z_{5}\}\},$$ where $Q=\{z_{1}, z_{2}, z_{3}, z_{4}, z_{5}\}$. Then
\smallskip
$$\begin{tabular}{lccccc}
 &\{$z_{1}$\}&\{$z_{2}$\}&\{$z_{1}, z_{3}$\}&\{$z_{2}, z_{3}, z_{4}$\}&\{$z_{1}, z_{3}, z_{4}, z_{5}$\}\}\\
 \cline{1-6}
\hline
$z_{1}$&1&0&1&0&1\\
 \cline{1-6}
$z_{2}$&0&1&0&1&0\\
 \cline{1-6}
$z_{3}$&0&0&1&1&1\\
 \cline{1-6}
$z_{4}$&0&0&0&1&1\\
$z_{5}$&0&0&0&0&1\\
 \cline{1-6}
\hline
\end{tabular}$$
$$\downarrow$$
$$\begin{tabular}{lccccc}
 &\{$z_{1}, z_{3}$\}&\{$z_{2}, z_{3}, z_{4}$\}&\{$z_{1}, z_{3}, z_{4}, z_{5}$\}\}&\{$z_{1}$\}&\{$z_{2}$\}\\
 \cline{1-6}
\hline
$z_{3}$&1&1&1&0&0\\
 \cline{1-6}
$z_{1}$&1&0&1&1&0\\
 \cline{1-6}
$z_{2}$&0&1&0&0&1\\
 \cline{1-6}
$z_{4}$&0&1&1&0&0\\
$z_{5}$&0&0&1&0&0\\
 \cline{1-6}
\hline
\end{tabular}$$
$$\downarrow$$
$$\begin{tabular}{lccccc}
 &\{$z_{1}, z_{3}$\}&\{$z_{2}, z_{3}, z_{4}$\}&\{$z_{1}, z_{3}, z_{4}, z_{5}$\}\}&\{$z_{1}$\}&\{$z_{2}$\}\\
 \cline{1-6}
\hline
$z_{3}$&1&1&1&0&0\\
 \cline{1-6}
$z_{1}$&1&0&1&1&0\\
 \cline{1-6}
$z_{2}$&0&1&0&0&1\\
 \cline{1-6}
\hline
\end{tabular}$$

Since $T_{2j}+T_{3j}=1$ for each $j=1, 2, 3$, it follows that the set of atom primary items of $(Q, \mathscr{H})$ is $D=\{z_{1}, z_{2}\}\}$.
\end{ex}

\maketitle
\subsection{The applications of connectedness and pre-quotient mapping in knowledge spaces}
\
\newline
\indent In this subsection, we discuss some applications of connectedness and pre-quotient mapping in knowledge spaces. First, we have the following proposition.

\begin{prop}
If $(Q, \mathscr{H})$ is a knowledge space, then $(Q, \tau_{\mathscr{H}})$ is connected iff $\mathscr{H}\cap \overline{\mathscr{H}}=\{\emptyset, X\}$, where $\overline{\mathscr{H}}=\{K: Q\setminus K\in\mathscr{H}\}$.
\end{prop}

A knowledge space $(Q, \mathscr{H})$ is called {\it granular} if for
each state $L\in \mathscr{H}_{t}$ and $t\in Q$, there exists an atom $H$ at $t$ with $H\subset L$. By Theorem~\ref{t17}, the following proposition holds.

\begin{prop}
If $(Q, \mathscr{H})$ is a granular knowledge space and $(Q, \tau_{\mathscr{H}})$ is regular, then both $H$ and $Q\setminus H$ belong to $\mathscr{H}$ for each atom $H$. Thus $(Q, \tau_{\mathscr{H}})$ is not connected.
\end{prop}

Assume that $\sigma$ is a function
mapping from a non-empty set $Q$ into $2^{2^{Q}}$. We
say that$\sigma$ is an {\it attribution} if $\sigma(t)\neq\emptyset$ for every $t\in Q$.
The following theorem is obvious.

\begin{thm}
Let $(Q, \mathscr{H})$ be a knowledge space which is derived from the attribution $\sigma$ on $Q$. If there exist a proper subset $Q^{\prime}$ of $Q$ and $C_{t}\in\sigma(t)$ for each $t\in Q$ such that $\bigcup_{t\in Q^{\prime}}C_{t}\subset Q^{\prime}$ and $\bigcup_{t\in Q\setminus Q^{\prime}}C_{t}\subset Q\setminus Q^{\prime}$, then $(Q, \tau_{\mathscr{H}})$ is not connected.
\end{thm}

\begin{thm}
Let $(Q, \mathscr{H})$ be a knowledge space. Then the mapping $h: (Q, \mathscr{H})\rightarrow (Q^{\ast}, \mathscr{H}^{\ast})$ defined by $h(t)=t^{\ast}$ for each $t\in Q$ is a pre-quotient mapping, that is, the discriminative reduction $(Q^{\ast}, \mathscr{H}^{\ast})$ is a pre-quotient pre-topology of $(Q, \mathscr{H})$.
\end{thm}

\begin{proof}
For each $K^{\ast}\in \mathscr{H}^{\ast}$, we have $h^{-1}(K^{\ast})=K$. Moreover, let $h^{-1}(A)$ is open $(Q, \mathscr{H})$. Then there is a subfamily $\mathcal{H}\subseteq \mathscr{H}$ with $h^{-1}(A)=\bigcup\mathcal{H}$, hence $A=\bigcup_{H\in\mathcal{H}}h(H)=\bigcup_{H\in\mathcal{H}}H^{\ast}$ is open in $(Q^{\ast}, \mathscr{H}^{\ast})$. Therefore, $h$ a pre-quotient mapping.
\end{proof}

\maketitle
\section{Conclusions}
The theory of knowledge spaces (KST) was originally regarded as a mathematical framework for knowledge
training and assessment. In fact, the concept of a knowledge
spaces is a generalization of topological spaces. Indeed, it is equivalent to the notion of pre-topological spaces in the present paper.
Therefore, the language of topology and its established generalizations are
very useful for studying in the theory of knowledge
spaces. For example, Falmagne's discriminative resemble the
axiom of separation $T_{0}$ for topologies. His notion of quasi ordinal space is just the well-known Alexandroff space in topology.

In order to study the theory of knowledge spaces, we systematically introduce the theory of
pre-topological spaces, such as, pre-base, subspace,
axioms of separation, connectedness, etc. Then we discuss the applications of axioms of separation for knowledge spaces. For a bi-discriminative knowledge space, we prove that the necessary and sufficient condition of an item belongs to some inner fringe or outer fringe of some knowledge state. Moreover, we give a characterization of a skill multimap such that answer to a problem in \cite{falmagne2011learning} or \cite{XGLJ} whenever each item with finitely many competencies. Further, we discuss the relation of Alexandroff spaces and quasi ordinal spaces. Indeed, the concept of Alexandroff spaces in topology is equivalent to the quasi ordinal spaces in knowledge spaces. We prove that each bi-discriminative quasi ordinal space $(Q, \mathscr{H})$ has $\mathscr{H}=2^{Q}$, and give a characterization of quasi ordinal spaces that are connected. Finally, we study the roles of the density of pre-topological spaces in primary items for knowledge spaces. We define the concept of primary items in knowledge spaces, that is, one person in order to master an item,
she or he must master items. We give an algorithm to find the set of atom primary items for any finite knowledge space.

Our paper only takes a first step to study the language of pre-topology in the theory of knowledge spaces. The theoretical framework introduced here and, in particular,
the method of studying the theory of knowledge spaces from the topological views are expected to
prove useful for future work in this direction.


\begin{thebibliography}{99}
\bibitem{Generalized2002}\'{A}. Cs\'{a}sz\'{a}r, Generalized topology, generalized continuity. Acta Math. Hungar., 96(2002), 351--357.

\bibitem{Generalized2004}\'{A}. Cs\'{a}sz\'{a}r, Separation axioms for generalized topologies. Acta Math. Hungar., 104(2004), 63--69.

\bibitem{Generalized2008}\'{A}. Cs\'{a}sz\'{a}r, On generalized neigghorhood systems. Acta Math. Hungar., 121(2008), 395--400.

\bibitem{Generalized2009}\'{A}. Cs\'{a}sz\'{a}r, Product of generalized topologies. Acta Math. Hungar., 123(2009), 127--132.

\bibitem{Arenas1999Alexandroff}
F.G. Arenas, Alexandroff spaces. Acta. Math. Univ. Comenianae, LXVIII(1)(1999) 17-25.

\bibitem{cosyn2009note}
E. Cosyn, H.B. Uzun, Note on two sufficient axioms for a well-graded
knowledge space. Journal of Mathematical Psychology, 53(1)(2009) 40-42.

\bibitem{2009Danilov}
V.I. Danilov, Knowledge spaces from a topological point of view. Journal of Mathematical Psychology, 53(2009)510-517.

\bibitem{doignon1985spaces}
J.P. Doignon, J.C. Falmagne, Spaces for The Assessment
of Knowledge. International Journal of Man-Machine Studies, 23(2)(1985) 175-196.

\bibitem{doignon1997well}
J.P. Doignon, J.C. Falmagne, Well-graded families of relations. Discrete Mathematics, 173(1-3)(1997) 35-44.

\bibitem{doignon1999knowledge}
J.P. Doignon, J.C. Falmagne. Knowledge Spaces. Springer Berlin Heidelberg, 1999.

\bibitem{doignon2011knowledge}
J.P. Doignon, J.C. Falmagne. Learning Spaces. Springer-Verlag Berlin Heidelberg, 2011.

\bibitem{falmagne1989latent}
J.C. Falmagne, A latent trait theory via a stochastic learning theory for a knowledge space. Psychometrika, 54(2)(1989) 283-303.

\bibitem{falmagne1990introduction}
J.C. Falmagne, M. Koppen, M. Villano, J.P. Doignon, L. Johanessen, Introduction to knowledge spaces: how to build, test and search them. Psychological Review, 97(1990) 204-224.

\bibitem{falmagne2011learning}
J.C. Falmagne, J.P. Doignon, Learning spaces: Interdisciplinary applied mathematics. Berlin, Heidelberg: Springer, 2011.

\bibitem{falmagne2013knowledge}
J.C. Falmagne, D. Albert, C. Doble, D. Eppstein, X. Hu, Knowledge spaces: Applications in education. Springer Science and Business Media, 2013.

\bibitem{falmagne1988markovian}
J.C. Falmagne, J.P. Doignon, A Markovian procedure for assessing the state of a system. Journal of Mathematical Psychology, 32(3)(1988b) 232-258.

\bibitem{E1989} R. Engelking, General Topology(revised and completed edition), Heldermann Verlag,
Berlin, 1989.

\bibitem{XGLJ}
X. Ge, J. Li, A note on the separability of items in knowledge structures delineated
by skill multimaps. Journal of Mathematical Psychology, 98 (2020) 102427.

\bibitem{lijinjin2004}
J. Li, Topological methods on the theory of covering generalized rough sets. PR\& AI, 17(1)(2004) 7-10.

\bibitem{lijinjin2006}
J. Li, The interior opeerators and the closure operators generated by a subbase. Adv. in Math., 35(4)(2006) 476-484.

\bibitem{lijinjin2007}
J. Li, The connectedness relative to a subbase for the topology. Adv. in Math., 36(4)(2007) 421-428.

\bibitem{liudejin2011}
D. Liu, The separateness relative to a subbase for the topology. College Math., 27(3)(2011) 59-65.

\bibitem{liudejin2013}
D. Liu, The regular space and relative regularity with regard to a subbase. Pure. Appl. Math., 29(6)(2013) 559-564.
\end{thebibliography}
\end{document}